\documentclass{amsart}

\usepackage{amssymb, amsmath}
\usepackage{mathrsfs}
\usepackage{amscd}
\usepackage{verbatim}

\allowdisplaybreaks

\theoremstyle{plain}
\newtheorem{theorem}{Theorem}[section]
\theoremstyle{remark}
\newtheorem{remark}[theorem]{\textbf{Remark}}
\newtheorem{example}[theorem]{\textbf{Example}}
\theoremstyle{plain}

\newtheorem{lemma}[theorem]{Lemma}
\newtheorem{proposition}[theorem]{Proposition}
\newtheorem{definition}[theorem]{Definition}

\numberwithin{equation}{section}

\def\N{{\mathbb N}}

\def\R{{\mathbb R}}
\def\C{{\mathbb C}}

\newcommand{\E}{{\mathbb E}}
\renewcommand{\P}{{\mathbb P}}
\newcommand{\F}{{\mathcal F}}

\newcommand{\g}{\gamma}

\newcommand{\e}{\varepsilon}

\renewcommand{\O}{\Omega}

\newcommand{\calL}{{\mathcal B}}
\newcommand{\n}{\Vert}
\newcommand{\one}{{{\bf 1}}}

\newcommand{\lb}{\langle}
\newcommand{\rb}{\rangle}

\newcommand{\limn}{\lim_{n\to\infty}}

\newcommand{\tE}{{\widetilde E}}
\newcommand{\dbn}{|\!|\!|}

\begin{document}

\title
[Non-autonomous stochastic evolution equations]{Non-autonomous stochastic
evolution equations and applications to stochastic partial differential
equations}

\author{Mark Veraar}
\address{Delft Institute of Applied Mathematics\\
Delft University of Technology \\ P.O. Box 5031\\ 2600 GA Delft\\The
Netherlands} \email{mark@profsonline.nl, M.C.Veraar@tudelft.nl}

\thanks{The author is supported by the Alexander von Humboldt foundation}









\date\today

\begin{abstract}

In this paper we study the following non-autonomous stochastic
evolution equation on a Banach space $E$,
\begin{equation}\label{eq:SEab}\tag{SE}
\left\{\begin{aligned}
dU(t) & = (A(t)U(t) + F(t,U(t)))\,dt + B(t,U(t))\,dW_H(t), \quad t\in [0,T],\\
 U(0) & = u_0.
\end{aligned}
\right.
\end{equation}
Here $(A(t))_{t\in [0,T]}$ are unbounded operators with domains
$(D(A(t)))_{t\in [0,T]}$ which may be time dependent. We assume that
$(A(t))_{t\in [0,T]}$ satisfies the conditions of Acquistapace and Terreni. The
functions $F$ and $B$ are nonlinear functions defined on certain interpolation
spaces and $u_0\in E$ is the initial value. $W_H$ is a cylindrical Brownian
motion on a separable Hilbert space $H$. We assume that the Banach space $E$ is
a UMD space with type $2$.

Under locally Lipschitz conditions we show that there exists a unique local
mild solution of \eqref{eq:SEab}. If the coefficients also satisfy a linear
growth condition, then it is shown that the solution exists globally. Under
assumptions on the interpolation spaces we extend the factorization method of
Da Prato, Kwapie\'n, and Zabczyk, to obtain space-time regularity results for
the solution $U$ of \eqref{eq:SEab}. For Hilbert spaces $E$ we obtain a maximal
regularity result. The results improve several previous results from the
literature.

The theory is applied to a second order stochastic partial differential
equation which has been studied by Sanz-Sol\'e and Vuillermot. This leads to
several improvements of their result.

\end{abstract}

\maketitle

\section{Introduction}

Let $E$ be a Banach space and $H$ be a separable Hilbert space. Let
$(\O,\F,\P)$ be a complete probability space with a filtration $(\F_t)_{t\in
[0,T]}$. In this paper we study the following stochastic evolution equation on
$E$:
\begin{equation}\label{eq:SEabintro}\tag{SE}
\left\{\begin{aligned}
dU(t) & = (A(t)U(t) + F(t,U(t)))\,dt + B(t,U(t))\,dW_H(t), \quad t\in [0,T],\\
 U(0) & = u_0.
\end{aligned}
\right.
\end{equation}
Here the operators $(A(t))_{t\in [0,T]}$ are unbounded and have domains
$(D(A(t)))_{t\in [0,T]}$ which may be time dependent. The functions
$F:[0,T]\times\O\times E\to E$ and $B:[0,T]\times\O\times E\to \calL(H,E)$ are
measurable and adapted functions and locally Lipschitz in a suitable way.
$W_H$ is a cylindrical Brownian motion with respect to $(\F_{t})_{t\in [0,T]}$
on a separable Hilbert space $H$. $u_0$ is an $\F_0$-measurable initial value.

Since the seventies, the problem \eqref{eq:SEabintro} has been studied by many
authors. We cannot give a complete description of the literature, but let us
give references to some selection of papers.

The method based on monotonicity of operators of \cite{Lio69} has
been applied to \eqref{eq:SEabintro} for instance in \cite{KR79} by
Krylov and Rozovski{\u\i} and in \cite{Par75,Par2} by Pardoux. We
will not discuss this method in more detail. For this we refer to
the monograph \cite{Rozov} of Rozovski{\u\i}.

In \cite{Da75}, Dawson used semigroup methods to study \eqref{eq:SEabintro} in
the autonomous case ($A$ is constant). This work has been further developed by
Da Prato and Zabczyk and their collaborators (cf.\ \cite{DPZ,DPZ92} and
references therein). In \cite{Sei}, Seidler considered the non-autonomous case
with $D(A(t))$ constant in time. In the above mentioned works, the authors
mainly considered their equation in a Hilbert space $E$. In \cite{Brz1,Brz2}
Brze{\'z}niak considered the autonomous case of \eqref{eq:SEabintro} in a UMD
space $E$ with type $2$ space (or even in martingale type $2$ spaces $E$).
This allows one to consider \eqref{eq:SEabintro} in $L^p$-spaces with $p\in
[2,\infty)$. Recently in \cite{NVW3}, van Neerven, Weis and the author
considered the autonomous case of \eqref{eq:SEabintro} in Banach spaces $E$
which include all $L^p$-spaces with $p\in [1,\infty)$. In \cite{VeZi}
Zimmerschied and the author studied \eqref{eq:SEabintro} with additive noise on
a general Banach space, and some parts of the current paper build on these ideas.

There are also many important papers where only $L^p$-spaces are considered.
Note that all of them always have the restriction that $p\in [2, \infty)$. Let
us first mention the works of Krylov and collaborators (see \cite{Kry} and
references therein). In these papers the authors use sophisticated methods
from partial differential equations and probability theory to obtain strong
space-regularity results for non-autonomous equations. Usually only second
order equations are considered and the methods are not based on semigroup
techniques. We explain some papers which use $L^p$-methods and semigroup
methods. In the paper of Manthey and Zausinger \cite{MZ} (also see their
references) $L^p$-methods and comparison methods are used to obtain global
existence results for the case where $F$ is non-necessarily of linear growth.
Let us mention that they also allow $D(A(t))$ to depend on time. However, they
do not give a systematic study of space-time regularity results. We believe it
is important to extend the ideas from \cite{MZ} to our general framework. This
could lead to interesting new global existence results. Also Cerrai
\cite{Cer}, Sanz-Sol\'e and Vuillermot \cite{SSV2,SSV}, and Zhang \cite{Zha}
consider $L^p$-methods. The papers \cite{SSV2,SSV} were the starting point of
our paper. The equation in \cite{SSV2,SSV} is a second order equation with
time-dependent boundary conditions. Below we consider it as our model problem.
%

In this paper we give a systematic theory for parabolic semi-linear stochastic
evolution equations, where $D(A(t))$ depends on time. It seems that such a
systematic study is new even in the Hilbert space setting. We study the
equation \eqref{eq:SEabintro} in a UMD space $E$ with type $2$. This class of
spaces includes all $L^p$-spaces with $p\in [2, \infty)$.
Although, a stochastic integration theory for processes with values in a general UMD is available \cite{NVW1}, we restrict ourselves to spaces with type $2$ in order to have a richer class of integrable processes (cf.\ Proposition \ref{prop:stochinttype2}). Note that the theory of \cite{NVW1} was applied in \cite{NVW3} for general UMD spaces, but only for autonomous equations. In order to consider nonautonomous equations
it seems that one needs additional assumptions on $A(t)$, and due to the extra technical
difficulties we will not consider this situation here.

Throughout the paper we assume that $(A(t))_{t\in [0,T]}$ satisfies the conditions of
Acquistapace and Terreni (AT1) and (AT2) (cf.\ \cite{AT2} and Section
\ref{sec:parabolic} below). These conditions are well-understood and widely
used in the literature. Let us mention that our results generalizes the main
setting of \cite{Brz2,DPZ92,Sei} in several ways. To prove regularity of the
solution we extend the factorization method of Da Prato, Kwapie\'n, and
Zabczyk. This well-known method gives space-time regularity of stochastic
convolutions. Compared to the known results, the main difficulty in our
version of the factorization method is that $D(A(t))$ is time dependent. For
Hilbert space $E$ we obtain a maximal regularity result. This extends the
result \cite[Theorem 6.14]{DPZ} to the non-autonomous case. The main tool in
our approach to maximal regularity is McIntosh's $H^\infty$-calculus
\cite{McI}.

To avoid technicalities at this point we will explain one of our
main results in a simplified setting. Assume the functions $F$ and
$B$ defined on $E$ are Lipschitz uniformly in $[0,T]\times\O$ (see \ref{as:LipschitzFtype} and \ref{as:LipschitzBtype} in Section \ref{sec:abstracteq} where a more general situation is considered).
In Section \ref{sec:Lipcoefinttype} we show that
\eqref{eq:SEabintro} has a unique mild solution. A strongly
measurable and adapted process $U:[0,T]\times\O\to E$ is called a
{\em mild solution} if for all $t\in [0,T]$, almost surely
\[U(t) = P(t,0) u_0 + P*F(\cdot,U)(t) + P\diamond B(\cdot,U)(t).\]
Here $(P(t,s))_{0\leq s\leq t\leq T}$ denotes the evolution system generated by
$(A(t))_{t\in [0,T]}$ and
\[P*F(t) = \int_0^t P(t,s)F(s,U(s))\, ds, \ \ P\diamond B(t) = \int_0^t P(t,s)B(s,U(s))\,
dW_H(s).\]
In Section \ref{sec:abstracteq} we also introduce so-called variational solutions in a general setting and show that they are equivalent to mild solutions.

We state a simplified formulation of one of our main results Theorem
\ref{thm:mainexistencegeninitial}. The hypothesis (AT1) and (AT2) are
introduced in Section \ref{sec:parabolic}. Hypothesis \ref{as:isom} is
introduced on page \pageref{as:isom}, and Hypotheses \ref{as:LipschitzFtype}
and \ref{as:LipschitzBtype} can be found in Section \ref{sec:abstracteq}.

\begin{theorem}\label{thm:mainexistencegeninitialintro}
Assume (AT1), (AT2), \ref{as:isom}, \ref{as:LipschitzFtype} and
\ref{as:LipschitzBtype} with $a=\theta=0$. Let $u_0:\O\to E$ be strongly $\F_0$
measurable. Then the following assertions hold:
\begin{enumerate}
\item There exists a unique mild solution $U$ of \eqref{eq:SEabintro} with
paths in $C([0,T];E)$ almost surely.

\item If $u_0\in (E,D(A(0))_{\eta,2}$ for some $\eta\in [0,\frac12]$, then for every
$\delta,\lambda>0$ with $\delta+\lambda <\eta$ there exists a version of $U$
with paths in $C^\lambda([0,T];\tE_{\delta})$.
\end{enumerate}
\end{theorem}
Here $(E,D(A(0))_{\eta,2}$ denotes real interpolation between $E$ and
$D(A(0))$. However, one may also take other interpolation spaces. One may think
of $\tE_{\delta}$ as time-independent version of $(E,D(A(t))_{\eta,2}$ (cf. \ref{as:isom} on page \pageref{as:isom}).

Actually in Section \ref{sec:abstracteq}, we will allow $F$ and $B$ which are defined on suitable interpolation
spaces and take values in certain extrapolation spaces. This enables
us to consider a larger class of noises. Moreover, in Section
\ref{sec:local} we even consider the case that $F$ and $B$ are
locally Lipschitz and Theorem \ref{thm:mainexistencegeninitialintro} has a version for locally
Lipschitz coefficients (see Theorem
\ref{thm:mainexistenceLocallyLipschitzCasetype}). It is also shown there that if
additionally $F$ and $B$ satisfy a linear growth condition as well, then the
full statements (1) and (2) of Theorem \ref{thm:mainexistencegeninitialintro}
still hold in the locally Lipschitz case.

\medskip

Our model equation is a problem which has been studied in \cite{SSV2,SSV}. Here
a second order order equation with time dependent boundary conditions is
considered. Sanz-Sol\'e and Vuillermot use a version of the factorization
methods to obtain existence, uniqueness and regularity results. Their methods
are based on estimates for Green's functions. They also consider two types of
variational solutions and mild solutions, and they show that these are all
equivalent. We obtain existence, uniqueness and regularity by applying the
above abstract framework. This leads to several improvements of
\cite{SSV2,SSV}. For example our space-time regularity results are better (see
Remark \ref{rem:compSSV}). We also show that our variational and mild solutions coincide with their solution concepts. Our setting seems more robust to adjustments of the equation (see Remark \ref{rem:othercases} and Example \ref{ex:loclLip}).

The stochastic partial differential equation is:
\begin{equation}\label{neumannintro}
\begin{aligned}
d u(t,s) &= A(t, s, D) u(t, s) + f(u(t,s)) \, dt \\ &\qquad \qquad
 \qquad + g(u(t,s)) \, d W(t,s), \ \ t\in (0,T], s\in S,
\\ C(t,s,D) u(t,s) &= 0, \ \ t\in (0,T], s\in \partial S
\\ u(0, s) &= u_0(s), \ \ s\in S.
\end{aligned}
\end{equation}
Here $S$ is a bounded domain with boundary of class $C^2$ and outer normal
vector $n(s)$ in $\R^n$, and
\begin{align*}
A(t,s,D) &= \sum_{i,j=1}^n D_i \Big( a_{ij}(t,s) D_j \Big) + a_0(t,s),
\\  C(t,s,D) &= \sum_{i,j=1}^n a_{ij}(t,s) n_i(s) D_j,
\end{align*}
where the coefficients $a_{ij}$ and $a_0$ are real valued and smooth and such
that $A(t,s,D)$ is uniformly elliptic (cf.\ Example \ref{ex:Lip1}). The
functions $f$ and $g$ are Lipschitz functions and $u_0$ is some
$\F_0$-measurable initial value. $W$ is a Brownian motion which is white with
respect to the time variable and colored with respect to the space variable.
More precisely in Example \ref{ex:Lip1} we will assume that the covariance
$Q\in \calL(L^2(S))$ of $W(1)$ satisfies $\sqrt{Q}\in
\calL(L^2(S),L^\infty(S))$.

In Example \ref{ex:Lip1} we will show the following consequence of Theorem
\ref{thm:mainexistencegeninitialintro}. For details we refer to Section
\ref{sec:SSV}.
\begin{enumerate}
\item Let $p\in [2,\infty)$. If $u_0\in L^p(S)$ a.s., then there exists a
unique mild and variational solution $u$ of \eqref{neumannintro} with paths in
$C([0,T];L^p(S))$ a.s. Moreover, $u\in L^2(0,T;W^{1,2}(S))$ a.s., where
$W^{1,2}(S)=H^1(S)$ is the Sobolev space.

\item If $u_0\in C^{1}(\overline{S})$ a.s., then the solution $u$ is in
$C^{\lambda}([0,T];C^{2\delta}(S))$ for all $\lambda,\delta>0$ such that
$\lambda+\delta<\frac12$. In particular, $u\in
C^{\beta_1,\beta_2}(\overline{S}\times [0,T])$  for all $\beta_1\in (0,1)$ and
$\beta_2\in (0,\frac12)$.
\end{enumerate}

The definition of a variational solution is given in Section
\ref{sec:abstracteq} (also see Remark \ref{rem:variationex}). The
definition of $C^{\beta_1,\beta_2}$ etc.\ can be found in Section
\ref{sec:local}.
In Example \ref{ex:Lip2} we will also obtain a version of the above result for the case $\sqrt{Q}\in \calL(L^2(S),L^q(S))$ for some
$q\in (1,\infty)$. In Example \ref{ex:loclLip} we show how to obtain a version
of the above result for locally Lipschitz coefficients $f$ and $b$.

\medskip

One can also study partial differential equations driven by multiplicative
{\em space-time white noise} using \eqref{eq:SEabintro}. For second order equations,
this is only possible for dimension one, and therefore not very illustrative
for our setting. In higher dimensions this seems to be possible if the order
of the operator is larger than the dimension. This has been considered in
\cite{NVW3} for the autonomous case (also see \cite{Brz2}). In the
non-autonomous setting the case of Dirichlet boundary conditions has been
studied in \cite[Chapter 8]{VThesis}. Some technical details have to be
overcome in order to treat the case of more general boundary conditions. Our results also have interesting consequences for stochastic partial differential equations with boundary noise. This is
work in progress \cite{VSchn}.

The papers is organized as follows. In Section \ref{sec:prel} we
discuss the preliminaries on evolution families, $H^\infty$-calculus and stochastic
integration theory. In Sections \ref{sec:detconv} and
\ref{sec:stochconv} we study space-time regularity of deterministic
and stochastic convolutions respectively. For this we extend the
factorization method for stochastic convolutions. We also prove a
maximal regularity result. The abstract stochastic evolution
equation will be given in Section \ref{sec:abstracteq}. Here we also
introduce variational and mild solutions. In Section
\ref{sec:Lipcoefinttype} we construct a unique mild solution of
\eqref{eq:SEabintro} by fixed-point methods under Lipschitz
conditions on the coefficients. The results are extended to the
locally Lipschitz case in Section \ref{sec:local}. Finally, in
Section \ref{sec:SSV} we consider the example \eqref{neumannintro}.

\section{Preliminaries\label{sec:prel}}

Below, we will use several interpolation methods (cf.\ \cite{Tr1} for
details). Let $(E_1,E_2)$ be an interpolation couple. For $\eta\in
(0,1)$ and $p\in [1, \infty]$, $(E_1,E_2)_{\eta,p}$ is the real
interpolation space between $E_1$ and $E_2$. Secondly,
$[E_1,E_2]_{\theta}$ is the complex interpolation between $E_1$ and
$E_2$.

We write $a \lesssim_K b$ to express that there exists a constant
$c$, only depending on $K$, such that $a\le c b.$ We write
$a\eqsim_K b$ to express that $a \lesssim_K b$ and $b \lesssim_K a$.
If there is no danger of confusion we just write $a \lesssim b$ for convenience.

\subsection{Parabolic evolution families}\label{sec:parabolic}
Let $(A(t),D(A(t)))_{t\in [0,T]}$ be a family of closed and densely defined
linear operators on a Banach space $E$. Consider the non-autonomous Cauchy
problem:
\begin{equation}\label{nCP}
\begin{aligned}
u'(t) &= A(t) u(t), \ \ t\in [s,T],
\\ u(s) &= x.
\end{aligned}
\end{equation}
We say that $u$ is a {\em classical solution} of \eqref{nCP} if $u\in C([s,T]; E) \cap
C^1((s,T]; E)$, $u(t)\in D(A(t))$ for all $t\in (s,T]$, $u(s) = x$, and
$\frac{d u}{dt}(t) = A(t) u(t)$ for all $t\in (s,T]$. We call $u$ a {\em strict
solution} of \eqref{nCP} if $u\in
C^1([s,T]; E)$, $u(t)\in D(A(t))$ for all $t\in [s,T]$, $u(s) = x$, and
$\frac{d u}{dt}(t) = A(t) u(t)$ for all $t\in [s,T]$.

A family of bounded operators $(P(t,s))_{0\leq s\leq t\leq T}$ on
$E$ is called a {\em strongly continuous evolution
family} if
\begin{enumerate}
\item $P(s,s) = I$  for all $s\in [0,T]$.
\item $P(t,s) = P(t,r) P(r,s)$ for all $0\leq s\leq r\leq t\leq T$.
\item The mapping $\{(\tau,\sigma)\in [0,T]^2: \sigma\leq \tau\}\ni (t,s) \to P(t,s)$ is strongly
continuous.
\end{enumerate}

We say that such a family $(P(t,s))_{0\leq s\leq t\leq T}$ {\em solves
\eqref{nCP} (on $(Y_s)_{s\in [0,T]}$)} if $(Y_s)_{s\in [0,T]}$ are dense
subspaces of $E$ such that for all $0\leq s\leq t\leq T$, we have $P(t,s)
Y_s\subset Y_t\subset D(A(t))$ and the function $t\mapsto P(t,s) x$ is a strict
solution of \eqref{nCP} for every $x \in Y_s$. In this case we say that
$(A(t))_{t\in [0,T]}$ generates the evolution family $(P(t,s))_{0\leq s\leq
t\leq T}$.

Well-posedness (i.e.\ existence, uniqueness, and continuous dependence on
initial values from $(Y_s)_{s\in [0,T]}$) of \eqref{nCP} is equivalent to the
existence and uniqueness of a strongly continuous evolution family that solves
\eqref{nCP} on $(Y_s)_{s\in [0,T]}$ (see \cite{Ni, Ni2} and the references
therein).
%
In the literature many sufficient conditions for this can be found, both in the
hyperbolic and parabolic setting (cf.\ \cite{AT2, Am, Lun87, Lun, Pa, Ta1,
Ta2,Ya} and the references therein). Below we will recall the parabolic setting
of \cite{AT2,Ya}.

If $E$ is a real Banach space everything below should be understood for the
complexification of the objects under consideration. First we recall some
results on sectorial operators.

Assume that for a closed operator $(A, D(A))$ there exist constants $M,w\geq
0$ and $\phi\in (\pi/2,\pi]$ such that $\Sigma(\phi,w)\subset \rho(A)$ and
\begin{equation}\label{eq:analytic}
\|R(\lambda, A)\|\leq \frac{M}{1+|\lambda-w|},  \ \  \lambda \in
\Sigma(\phi,w).
\end{equation}
Here $\Sigma(\phi, w) = \{w\}\cup \{\lambda\in \C\setminus \{w\}:
|\arg (\lambda -w)|\leq \phi\}$. We denote $A_w = A-w$.

It is well-known that by \eqref{eq:analytic}, $A$ generates an
analytic semigroup. In this case for $\delta>0$ one can define
$(-A_w)^{-\delta}\in \mathcal{B}(E)$ by
\[(-A_w)^{-\delta} = \frac{1}{2\pi i} \int_{\Gamma} (w-\lambda)^{-\delta} R(\lambda, A) \, d\lambda,\]
where the contour $\Gamma = \{\lambda: \arg(\lambda-w) = \pm \phi\}$ is
orientated counter clockwise (cf.\ \cite{Am, Lun, Pa, Ta1} for details).
Furthermore, recall that the operator $(w-A)^{\delta}$ is defined as the
inverse of $(w-A)^{-\delta}$. For all $\beta>\alpha$,
\begin{equation}\label{eq:embedd}
(E, D(A))_{\beta, \infty}\hookrightarrow (E, D(A))_{\alpha,
1}\hookrightarrow D((w-A)^{\alpha}) \hookrightarrow (E,
D(A))_{\alpha, \infty},
\end{equation}
where embedding constants only dependent on $\alpha, \beta$ and the
constants in \eqref{eq:analytic}.


As before, let $(A(t),D(A(t)))_{t\in [0,T]}$ be a family of closed and densely
defined operators on a Banach space $E$. We will briefly discuss the setting of
Acquistapace and Terreni \cite{AT2}. Note that most of the results below have
versions for non-densely defined $A(t)$ as well. In fact they study a slightly
more general setting.

Condition (AT) is said to be satisfied if the following two
conditions hold:
\begin{enumerate}
\item[(AT1)]  \label{AT1}
$A(t)$ are linear operators on a Banach space $ E $ and there are
constants $w\in \R$, $K \ge 0$, and $ \phi \in (\frac{\pi}{2},\pi)$
such that $\Sigma(\phi,w) \subset \varrho(A(t))$ and for all $
\lambda \in \Sigma(\phi,w)$ and $t\in [0,T]$,
\[ \| R(\lambda, A(t)) \| \le \frac{K}{1+ |\lambda-w|}.\]
\item[(AT2)]  \label{AT2}
There are constants $L \ge 0$ and $\mu, \nu \in (0,1]$ with $\mu +
\nu >1$ such that for all $\lambda \in \Sigma(\phi,0)$ and $s,t \in
[0,T]$,
\[ \| A_w(t)R(\lambda,
A_w(t))(A_w(t)^{-1}- A_w(s)^{-1})\| \le L |t-s|^\mu
(|\lambda|+1)^{-\nu}.  \]
\end{enumerate}
Below it will be convenient to denote $\kappa_{\mu,\nu} =
\mu+\nu-1\in (0,1]$.

These conditions have been extensively studied in the literature,
where also many examples can be found. The first condition may be
seen as analyticity uniformly in $t \in [0,T]$.


If (AT1) holds and the domains are constant $D(A(0))=D(A(t))$, $t\in [0,T]$,
then H\"older continuity of $(A(t))_{t\in [0,T]}$ in $\calL(D(A(0)),E)$ with
exponent $\eta$, implies (AT2) with $\mu = \eta$ and $\nu=1$ (see \cite[Section
7]{AT2}). The conditions in that case reduce to the conditions in the theory of
Sobolevski\u\i \, and Tanabe for constant domains (cf.\ \cite{Lun, Pa, Ta1}).

We will use the notation  $E_\eta^t = (E, D(A(t)))_{\eta, 2}$ for
$t\in [0,T]$ unless it is stated otherwise. Also see \ref{as:isom1a}
on page \ref{as:isom1a}. Further, we write that $E^t_{-\theta}$ for
the completion of $E$ with respect to the norm
$\|x\|_{E^t_{-\theta}} = \|(-A_w(t))^{-\theta}x\|$.

Under the assumptions (AT1) and (AT2) the following result holds
(see \cite[Theorems 6.1-6.4]{AT2} and \cite[Theorem 2.1]{Ya}).
\begin{theorem} \label{thm:exist-parab}
If condition (AT) holds, then there exists a unique strongly continuous
evolution family $(P(t,s))_{0\le s \le t \le T}$ that solves \eqref{nCP} on
$D(A(s))$ and for all $x\in E$, $P(t, s) x$ is a classical solution of
\eqref{nCP}. Moreover, $(P(t,s))_{0\le s \le t \le T}$ is continuous on $0 \le
s < t \le T$ and there exists a constant $C>0$ such that for every $0\leq
s<t\leq T$,
\begin{eqnarray}\label{eq:2_13}
\|P(t,s) x\|_{E_{\alpha}^t} & \le & C(t-s)^{ \beta-\alpha}\|x\|_{E^{s}_{\beta}}
\qquad \mbox{ for } 0 \leq\beta\leq\alpha \le 1,
\end{eqnarray}
\end{theorem}

%


We recall from \cite[Theorem 2.1]{Ya} that there is a constant $C>0$
such that for all $\theta\in (0,\mu)$ and for all $x\in
D((w-A(s))^{\theta})$,
\begin{equation}
\label{eq:2_15} \| P(t,s)( w-A(s))^\theta x\|  \le  C (\mu-
\theta)^{-1}(t-s)^{- \theta} \|x\|.
\end{equation}

%

Consider the following Hypothesis.

\let\ALTERWERTA\theenumi
\let\ALTERWERTB\labelenumi
\def\theenumi{(H1)$'_{\eta_0}$}
\def\labelenumi{(H1)$'_{\eta_0}$}
\begin{enumerate}
\item \label{as:isom1a} There exists an $\eta_0\in (0,1]$ and an family of
spaces $(\tE_{\eta})_{\eta\in [0,\eta_0]}$ such that
\[\tE_{\eta_0} \hookrightarrow \tE_{\eta_1} \hookrightarrow \tE_{\eta_2} \hookrightarrow \tE_{0}=E, \ \ \ 0\leq \eta_2\leq \eta_1\leq \eta_0.\]
and for all
$\eta\in [0,\eta_0]$
\[E_{\eta}^t := (E, D(A(t)))_{\eta, 2} \hookrightarrow \tE_{\eta}\hookrightarrow E\]
with uniform constants in $t\in[0,T]$.
\end{enumerate}
\let\theenumi\ALTERWERTA
\let\labelenumi\ALTERWERTB
Alternatively, one could replace $(E, D(A(t)))_{\eta, 2}$ by $(E,
D(A(t)))_{\eta, p}$ for $p\in (2, \infty)$ or by the complex interpolation
spaces $[E, D(A(t))]_{\eta}$.

Assumption \ref{as:isom1a} enables us to deduce space time regularity results.
Such type conditions are often used to get rid of the time dependence. In
applications this gives a way to obtain H\"older regularity in space. A similar
condition can be found in \cite[Hypothesis (H2)]{MSchnnew}. Later on we will
strengthen \ref{as:isom1a} to a condition \ref{as:isom} (see page \pageref{as:isom}). There we also require that the space $\tE_{\eta}$ are UMD and of type $2$. This is the main reason one can only allow $p\in [2, \infty)$ if one considers $(E,
D(A(t)))_{\eta, p}$.

%
%

In many examples one can take $\tE_{\eta} = E_{\eta}^t$ for $\eta$
small. For second order operators on $L^p$-spaces, \ref{as:isom1a}
is usually fulfilled for $\eta_0 =\frac12$. However, since it can be
difficult to calculate $E_{\eta}^t$ it will be convenient to work in
the above setting.
In the next example we briefly motivate why it is useful to consider the spaces
$\tE_\eta$.

\begin{example}\label{ex:boundary}
Consider a second order elliptic differential operator $A(t)$ on a domain $S$
with time dependent boundary condition $C(t) u =0$. If this is modeled
on $E=L^p(S)$ for $p\in [2,\infty)$, then one usually has $D(A(t)) = \{f\in
W^{2,p}(S): C(t) f =0\}$. Often one shows that the solution $u$ takes its
values in $E_{\eta}^t= (E, D(A(t)))_{\eta, 2}$. However, it may be difficult
to characterize $E_{\eta}^t$ because of the boundary condition. It is even not
clear whether such a space is time independent. This will be needed below. It
is easier to calculate $\tE_{\eta} = (E, W^{2,p}(S))_{\eta, 2}$, which is
$B^{2\eta}_{p,2}(S)$ for regular $S$. This space is time independent and
regularity in the space $B^{2\eta}_{p,2}(S)$ usually suffices.

Recall from Grisvard's result (cf.\ \cite[Theorem 4.3.3]{Tr1}) that for domains
and coefficients which are $C^{\infty}$ one can characterize the spaces
$E_{\eta}^t$ as certain subspaces of $B^{2\eta}_{p,2}(S)$. A similar result for complex interpolation spaces follows from Seeley \cite{Se}. In Amann \cite[Section 7]{Amnonhom} it is explained that for second order elliptic operators on $L^p$-spaces the boundary conditions in $D(A(t)))_{\eta, p}$ for $p\in (2, \infty)$ or $[E, D(A(t))]_{\eta}$ disappear under the natural conditions on $p$ and $\eta$. Although his assumptions on the domain seems to be that it is $C^\infty$, but it follows from \cite[Remark 7.3]{Amnonhom} and \cite[Theorem 2.3]{DDHPV} that a $C^2$ boundary suffices.
\end{example}

\begin{lemma}\label{lem:initialcons}
Assume (AT) and \ref{as:isom1a} for some $\eta_0\in (0,1]$.

Let $\alpha\in (0,\eta_0]$. Let $\delta,\lambda>0$ be such that
$\delta+\lambda\leq \alpha$. Then there exists a constant $C$ such that for
all $0\leq r\leq s\leq t\leq T$ and for all $x\in E_{\alpha}^r$
\begin{equation}\label{eq:holderx}
\begin{aligned} \|P(t,r)x - P(s,r)x\|_{\tE_{\delta}} &\leq C|t-s|^{\lambda}
\|x\|_{E_{\alpha}^r}.
\end{aligned}
\end{equation}

Moreover, if $\alpha\in [0,\eta_0)$ and $x\in E_{\alpha}^r$, then $t\mapsto
P(t,r)x\in C([r,T];\tE_\alpha)$.

\end{lemma}


\begin{remark}
Under additional assumption on $\alpha,\delta,\mu,\nu$, there is a
version of Lemma \ref{lem:initialcons} for the case that
$\alpha>\eta_0$. Since we will not need this in our examples, we
will not consider this situation.
\end{remark}

\begin{proof}
%

It follows from \eqref{eq:2_13} that
\begin{align*}
\|P(t,r)x - P(s,r)x\|_{\tE_{\alpha}}&\leq \|P(t,r)x\|_{\tE_{\alpha}} +
\|P(s,r)x\|_{\tE_{\alpha}} \\ & \lesssim \|P(t,r)x\|_{E_{\alpha}^t} +
\|P(s,r)x\|_{E_{\alpha}^s} \lesssim \|x\|_{E_{\alpha}^r}.
\end{align*}
Moreover, by \cite[(2.16)]{Schn} we obtain that
\[
\|P(t,r)x - P(s,r)x\| \leq C|t-s|^{\alpha} \|x\|_{E_{\alpha}^r}.
\]
Therefore, by interpolation with $\delta = \theta\alpha$ and
$\lambda = (1-\theta)\alpha$ for $\theta\in [0,1]$ we obtain
\[\|P(t,r)x - P(s,r)x\|_{\tE_{\delta}}\lesssim |t-s|^{\lambda} \|x\|_{E_{\alpha}^r}.\]
This proves the first part.

For the second part let $x\in E_{\alpha}^r$, and take $x_1, x_2 \ldots$ in
$E_{\eta_0}^r$ such that $x= \limn x_n$ in $E_{\alpha}^r$. Then the first
result shows that each $t\mapsto P(t,r)x_n$ in $\tE_{\alpha}$ is continuous.
Moreover, as before
\[\|P(t,r)x - P(t,r)x_n\|_{\tE_\alpha} = \|P(t,r)(x-x_n)\|_{\tE_\alpha}\lesssim \|x-x_n\|_{E_{\alpha}^r}.\]
Therefore, $P(t,r)x = \limn P(t,r)x_n$ in $\tE_{\alpha}$ uniformly in $t\in
[0,T]$, and it is continuous.
\end{proof}

\subsection{$H^\infty$-calculus on Hilbert spaces\label{subsec:Hinfty}}

In Section \ref{sec:stochconv} we use McIntosh's $H^\infty$-calculus in order to derive maximal regularity of stochastic convolutions on Hilbert spaces $E$. Here we briefly recall the definition and a characterization which gives the way in which we will use the $H^\infty$-calculus. For details we refer to \cite{ADM,Haase:2,KuWe,McI} and references therein. Although we only explain $H^\infty$-calculus on Hilbert spaces, there are extensions to Banach spaces. Our situation slightly differs from the existing literature in the sense that $A$ is replaced by $-A$ and we assume $0\in \rho(A)$. Moreover, we only consider analytic semigroup generators below.

Let $E$ be a Hilbert space and let $A$ be a closed, densely defined operator on $E$. Assume $A$ is sectorial of type $\phi\in (\pi/2, \pi]$, i.e.\ the sector $\Sigma_{\phi}:=\Sigma(\phi,0)\subset \rho(A)$
and there exist a constant $M$ such that for all
\[\|R(\lambda, A)\|\leq \frac{M}{1+|\lambda|},  \ \  \lambda \in
\Sigma_\phi.\]
The largest constant $\phi$ for which such an $M$ exists will be denoted with $\phi(A)$.

For $\sigma\in (0,\pi]$, let $H^\infty(\Sigma_\sigma)$ denote the
space of bounded analytic functions $f:\Sigma_\sigma\to \C$ with norm
$\|f\|_{H^\infty(\Sigma_\sigma)} = \sup_{\lambda\in \Sigma_\sigma}
|f(\lambda)|$. Let
\[
H^\infty_0(\Sigma_\sigma) = \Big\{f\in H^\infty(\Sigma_\sigma): \exists \epsilon>0 \ \text{s.t.} \
|f(\lambda)|\leq \frac{|z|^\e}{(1+|z|^2)^\e}\Big\}.
\]

Let $A$ be as above, and fix $\sigma\in (\pi/2,\phi(A))$ and $\theta\in (\sigma, \phi(A))$.
For $f\in H^\infty_0(\Sigma_\sigma)$, one can define
\[f(A) = \frac{1}{2\pi i}\int_{\partial \Sigma_{\theta}} f(\lambda) R(\lambda,A) \, d\lambda, \]
where the integral converges in the Bochner sense. We say that $A$
has a {\em bounded $H^\infty(\Sigma_\sigma)$-calculus for
$\sigma\in (\pi/2,\phi(A))$} if there is a constant $C$ such that
\begin{equation}\label{eq:Hinfty}
\|f(A)\|\leq C \|f\|_{H^\infty(\Sigma_{\sigma})} \ \ \text{for all $f\in
H^\infty_0(\Sigma_\sigma)$}.
\end{equation}
In this case \eqref{eq:Hinfty} has a unique continuous extension to
all $f\in H^\infty(\Sigma_{\sigma})$. The boundedness of the $H^\infty$-calculus is characterized by the following theorem.

\begin{proposition}
Let $E$ be a Hilbert space and let $A$ be as above. Then the following assertions are equivalent:
\begin{enumerate}
\item $A$ has a bounded $H^\infty(\Sigma_{\sigma})$-calculus for some (all) $\sigma\in (\pi/2, \phi(A))$.

\item $-A$ has bounded imaginary powers and for some (all) $\sigma\in (\pi/2, \phi(A))$ there is a constant $C>0$ such that for all $s\in \R$, $\|(-A)^{is}\|\leq C e^{\sigma |s|}$.

\item For some (all) $|\sigma|\in (-\phi(A), \phi(A))$ there exists a constant $C>0$ such that for all $x\in E$,
\begin{equation}\label{eq:squarefunctionResolvent}
C^{-1}\|x\|\leq \Big(\int_0^\infty \|(-A)^{\frac12} R(t e^{i\sigma},A)x\|^2 \, dt\Big)^{\frac12} \leq C\|x\|.
\end{equation}

\end{enumerate}
\end{proposition}
This result can be found in \cite{McI} (also see \cite[Theorem 11.9]{KuWe}). The estimate \eqref{eq:squarefunctionResolvent} is called a square function estimate.
Applying the Fourier transform to $t\mapsto R(t e^{i\sigma},A)x$ with $\sigma = \pi/2$ one obtains that there exists a constant $C_2$ such that for all $x\in E$,
C\begin{equation}\label{eq:squarefunction}
C_2^{-1}\|x\|\leq \Big(\int_0^\infty \|(-A)^{\frac12} e^{tA}x\|^2 \, dt\Big)^{\frac12} \leq C_2\|x\|.
C\end{equation}
The important estimate for us will be
\begin{equation}\label{eq:squarefuncest0}
\Big(\int_0^\infty \|(-A)^{\frac12} e^{tA}\|^2 \, dt\Big)^{\frac12} \leq C_2\|x\|.
\end{equation}
The same estimates as in \eqref{eq:squarefuncest0} hold for $A^*$.
Moreover, if \eqref{eq:squarefuncest0} holds
for $A$ and $A^*$, this again implies the boundedness of the $H^\infty$-calculus.
We further note that the estimate \eqref{eq:squarefuncest0} is also used for the
Weiss conjecture in control theory (cf.\ \cite{LeM} and references therein).

Not every sectorial operator $A$ has a bounded $H^\infty$-calculus
of some angle. Counterexamples are given in \cite{McY}.
However, many examples are known to have a bounded $H^\infty$-calculus.
We state some sufficient conditions for the boundedness of the $H^\infty$-calculus.
\begin{remark}\label{rem:Hinfty0}
For a Hilbert space $E$, each of the following conditions is sufficient for having a bounded $H^\infty$-calculus
\begin{enumerate}
\item $A$ generates an analytic contraction semigroup (see \cite[Theorem 11.13]{KuWe}).

\item $-A$ is positive and self-adjoint. In this case one has $C_2=1$ in \eqref{eq:squarefunction} (cf.\ \cite[Example 11.7]{KuWe}) and $C=1$ in \eqref{eq:Hinfty} (cf.\ \cite[Section (G)]{ADM}).
\end{enumerate}
\end{remark}


\subsection{$\gamma$-Radonifying operators and stochastic integration\label{sec:stochint}}
We recall some results on $\gamma$-radonifying operators and stochastic
integration. For details on the subject we refer to
\cite{Bog,Brz2,DJT,KaWe,NW1,NVW1}.

Let $E$ be a Banach space and $H$ be a separable Hilbert space. Let
$(S,\mu)$ be a measurable space. A function $\phi:S\to E$ is called
{\em strongly measurable} if it is the pointwise limit of a sequence
of simple functions.

Let $E_1$ and $E_2$ be Banach spaces. An operator-valued function $\Phi:S\to
\calL(E_1,E_2)$ will be called {\em $E_1$-strongly measurable} if for all $x\in
E_1$, the $E_2$-valued function $\Phi x$ is strongly measurable.

If $(S, \Sigma, \mu)$ is a measure space and $\phi:S\to E$ is defined as an
equivalence class of functions, then we say that $\phi$ is {\em strongly
measurable} if there is a version of $\phi$ which is strongly measurable.

A bounded operator $R\in\calL(H,E)$ is said to be a $\gamma$-radonifying
operator if there exists an orthonormal basis $(h_n)_{n\ge 1}$ of $H$ such
that $\sum_{n\geq 1} \g_n \, Rh_n$ converges in $L^2(\O;E)$. We then define
$$ \n R\n_{\g(H,E)}
:= \Bigl(\E\Bigl\n  \sum_{n\geq 1} \g_n \,Rh_n \Bigr\n^2\Bigr)^\frac12.$$ This
number does not depend on the sequence $(\g_n)_{n\ge 1}$ and the basis
$(h_n)_{n\ge 1}$, and defines a norm on the space $\gamma(H, E)$ of all
$\gamma$-radonifying operators from $H$ into $E$. Endowed with this norm,
$\g(H,E)$ is a Banach space, which is separable if $E$ is separable. Moreover,
$\n R\n\le \n R\n_{\g(H,E)}$.

If $E$ is a Hilbert space, then $\g(H,E) = \mathcal{C}^2(H,E)$ isometrically,
where $\mathcal{C}^2(H,E)$ is the space of Hilbert-Schmidt operators. Also for
$E=L^p$ there are easy characterization of $\gamma$-radonifying operators.
Such a characterization has been obtained in \cite{BrzvN03}. We use a slightly
different formulation from \cite{NVW3}.

\begin{lemma}\label{lem:sq-fc-Lp}
Let $(S,\Sigma,\mu)$ be a $\sigma$-finite measure space and let $1\le
p<\infty$. For an  operator $R\in\calL(H,L^p(S))$ the following assertions are
equivalent:
\begin{enumerate}
\item  $R\in \g(H,L^p(S))$.

\item  There exists a function $g\in L^p(S)$ such that for all $h\in H$ we have
$|R h|\le \|h\|_H\cdot g$ $\mu$-almost everywhere.
\end{enumerate}
Moreover, in this situation we have
\begin{equation}\label{eq:Rsquarekappa}
\|R\|_{\g(H,L^p(S))} \lesssim_p \|g\|_{L^p(S)}.
\end{equation}
\end{lemma}

Let $(r_n)_{n\geq 1}$ be a Rademacher sequence on $(\O,\F,\P)$. Recall that a Banach space $E$ is said to have {\em type $2$} if there exists
a constant $C\ge 0$ such that for all finite subsets $\{x_1,\dots,x_N\}$ of
$E$ we have
\begin{equation}\label{eq:type2est}
\Big(\E \Big\n \sum_{n=1}^N r_n x_n\Big\n^2\Big)^\frac12 \le C
\Big(\sum_{n=1}^N \n x_n\n^2\Big)^\frac12.
\end{equation}
Hilbert spaces have type $2$ and the $L^p$-spaces for $p\in[2,\infty)$ have
type $2$ (see \cite[Chapter 11]{DJT} for details). Furthermore, Besov spaces
$B^{\alpha}_{p,q}$ and Sobolev spaces $W^{\alpha,p}$ have type $2$ as long as
$2\leq p,q<\infty$. This follows from the fact that these spaces are
isomorphic to closed subspaces of $L^p$-spaces and $\ell^q(L^p)$-spaces
(cf.\ \cite{Tr1}).

We will also need UMD Banach spaces. The definition of a UMD space will be
omitted. We refer to \cite{Bu3} for an overview on the subject. Important
examples of UMD spaces are the reflexive scale of $L^p$, Sobolev and Besov
spaces. Moreover, we note that every UMD space is reflexive.

A detailed stochastic integration theory for operator-valued processes
$\Phi:[0,T]\times\O\to \calL(H,E)$ where $E$ is a UMD space, is given in
\cite{NVW1}. For our purposes it will be enough to consider UMD spaces with
type $2$. In this situation there is an easy subspace of integrable processes
which will be large enough for our considerations. Instead of UMD spaces with type $2$ one can also
assume that $E$ is a of martingale type $2$ space (cf.\ \cite{Brz2,Pi75}). We
do not consider this generality, because it is unnecessary for our applications to
stochastic partial differential equations.

A family $W_H=(W_H(t))_{t\in \R_+}$ of bounded linear operators from $H$ to
$L^2(\O)$ is called an {\em $H$-cylindrical Brownian motion} if
\renewcommand{\labelenumi}{(\roman{enumi})}
\renewcommand{\theenumi}{(\roman{enumi})}
\begin{enumerate}
\item $W_H h = (W_H(t)h)_{t\in \R_+}$ is a scalar-valued Brownian motion for
each $h\in H$, \item $ \E (W_H(s)g \cdot W_H(t)h) = (s\wedge t)\,[g,h]_{H}$ for
all $s,t\in \R_+, \ g,h\in H.$
\end{enumerate}
\renewcommand{\labelenumi}{(\arabic{enumi})}
\renewcommand{\theenumi}{(\arabic{enumi})}
We always assume that the $H$-cylindrical Brownian motion $W_H$ is with respect
to the filtration $(\F_t)_{t\geq 0}$, i.e., $W_Hh$ are Brownian motions with
respect to $(\F_t)_{t\geq 0}$ for all $h\in H$.

Now let $E$ be a UMD Banach space with type $2$. For an $H$-strongly
measurable and adapted $\Phi:[0,T]\times \O\to \g(H,E)$ which is in
$L^2((0,T)\times \O;\g(H,E))$ one can define the stochastic integral $\int_0^T
\Phi(s) \, d W_H(s)$ as a limit of integrals of adapted step processes and
(cf.\ \cite{Brz2} and references therein) there exists a constant $C$ not
depending on $\Phi$ such that
\begin{equation}\label{eq:type2estint}
(\E \Big\|\int_0^T \Phi(s) \, d W_H(s)\Big\|^2 )^{\frac12} \leq C
\|\Phi\|_{L^2((0,T)\times \O;\g(H,E))}.
\end{equation}
By a localization argument one may extend the class of integrable processes to
all $H$-strongly measurable and adapted $\Phi:[0,T]\times \O\to \g(H,E)$ which
are in $L^2(0,T;\g(H,E))$ a.s. Moreover, the estimate \eqref{eq:type2estint}
for the stochastic integral also implies type $2$.


In \cite{NVW1} two-sided estimates for the stochastic integral are given using
generalized square function norms. As a consequence of that theory one also
obtains the above results. The result that we will frequently use is the
following (cf.\ \cite{Brz2} and \cite[Corollary 3.10]{NVW1}).
\begin{proposition}\label{prop:stochinttype2}
Let $E$ be a UMD space with type $2$. Let $\Phi:[0,T]\times \O\to \g(H,E)$ be
$H$-strongly measurable and adapted. If $\Phi\in L^2(0,T;\g(H,E))$ a.s., then
$\Phi$ is stochastically integrable with respect to $W_H$ and for all $p\in (1,
\infty)$,
\[ \Big(\E \sup_{t\in
[0,T]}\Big\|\int_0^{t}\Phi(s)\,dW_H(s)\Big\|^p\Big)^{\frac{1}{p}}
\lesssim_{E,p} \|\Phi\|_{L^p(\O;L^2(0,T;\g(H,E)))}.
\]
\end{proposition}
Again the estimate in Proposition \ref{prop:stochinttype2} implies that $E$
has type $2$.

We will also use the following basic fact for $\Phi$: as in
Proposition \ref{prop:stochinttype2} for $x^*\in E^*$,
\[
\Big\lb \int_0^T \Phi(s) \, d W_H(s),x^*\Big\rb = \int_0^T \Phi(s)^* x^* \, d W_H(s) \  a.s.
\]


\section{Deterministic convolutions\label{sec:detconv}}

Let $E$ be a Banach space. For $\alpha\in (0,1]$, $p\in[1, \infty]$ and $f\in
L^p(0,T;E)$, define the function $R_{\alpha}f\in L^p(0,T;E)$ by
\begin{equation}\label{Ralphadef}
(R_{\alpha}f)(t) = \frac{1}{\Gamma(\alpha)}\int_0^t (t-s)^{\alpha-1}
P(t,s) f(s)\, ds.
\end{equation}
This is well-defined by Young's inequality and there is a constant
$C\geq 0$ that only depends on $\alpha, p$ and $\sup_{0\leq s\leq
t\leq T} \|P(t,s)\|$ such that
\[\|R_{\alpha} f\|_{L^p(0,T;E)}\leq C T^{\alpha}\|f\|_{L^p(0,T;E)}.\]

\begin{lemma}\label{fact2}
Assume that (AT) and \ref{as:isom1a} with some $\eta_0\in (0,1]$
hold. Let $\alpha \in (0, \eta_0]$, $\delta,\lambda \in [0,1)$, and
$p\in [1, \infty)$ be such that $\alpha - \frac{1}{p} -
\delta-\lambda> 0$. Then for every $f\in L^p(0,T;E)$,
$R_{\alpha}f\in C^\lambda([0,T];\tE_{\delta})$ and there is a
constant $C\geq 0$ such that for all $f\in L^p(0,T;E)$,
\[\|R_{\alpha}f\|_{C^{\lambda}([0,T];\tE_{\delta})}\leq C\|f\|_{L^p(0,T;E)}.\]
\end{lemma}

\begin{proof}
This can be proved in a similar way as in \cite[Lemma 4.1]{VeZi}, by replacing
the fractional domain spaces by $\tE_{\eta}$. The only part of the proof of
\cite[Lemma 4.1]{VeZi} that requires a different argument is the estimate for
\[I_3 = \Big\|\int_0^s (s-r)^{\alpha -1 } (P(t,r) - P(s,r)) f(r) \, dr\Big\|_{\tE_{\delta}}.\]
We have to show that $I_3 \lesssim |t-s|^{\lambda}$. It follows
from Lemma \ref{lem:initialcons} and \eqref{eq:2_13} that for $x
\in E$
\begin{align*}
\|P(t,r) x - P(s,r)x\|_{\tE_{\delta}} &= \|(P(t,s) - I)P(s,r)x\|_{\tE_{\delta}}
\\ & \lesssim |t-s|^{\lambda} \| P(s,r)x\|_{\tE_{\delta+\lambda}}
\\ & \lesssim |t-s|^{\lambda} \| P(s,r)x\|_{E_{\delta+\lambda}^r}
\\ & \lesssim |t-s|^{\lambda} (s-r)^{-\lambda-\delta}
\|x\|.
\end{align*}
This implies the estimate for $I_3$.
\end{proof}

Recall that $E^t_{-\theta}$ be the completion of $E$ with respect to
the norm $\|x\|_{E^t_{-\theta}} = \|(-A_w(t))^{-\theta}x\|$.

The next result will be formulated for a family $\{\phi(t):t\in [0,T]\}$ such
that for all $t\in [0,T]$, $\phi(t,\omega)\in E_{-\theta}^t$, where
$(-A_w)^{-\theta} \phi$ is a strongly measurable function from $[0,T]$ into $E$
and $\theta\in [0,1)$ is fixed. We denote the deterministic convolution by
\[P*\Phi(t) := \int_0^t P(t,s) \phi(s) \, d s,\]
where $\phi$ is as above.

First we explain some general measurability properties which hold under the
(AT) conditions. Let $\theta\in [0,\mu)$.
One has that for all $0\leq s<t\leq T$, $P(t,s)(w-A(s))^{\theta}$ has an
extension to an operator in $\calL(E)$ (see \eqref{eq:2_15}). We claim that as
a function of $(s,t)$ where $0\leq s<t\leq T$, this extension is $E$-strongly
measurable.

Indeed, let $A_n(t) = n A(t)R(n;A(t))$ be the Yosida approximations of $A(t)$.
Then for all $x\in E$ (see proof of \cite[Proposition 3.1]{Ya90}) one has
\[\limn P_n(t,s) (w-A_{n}(s))^{\theta} x = P(t,s) (w-A(s))^{\theta} x,\]
where $P_n(t,s)$ is the evolution family generated by $A_n(t)$. Since
$(t,s)\mapsto P_n(t,s) (w-A_{n}(s))^{\theta} x$ is strongly measurable, the
claim follows.

It follows that for $0\leq s<t\leq T$, $P(t,s)$ has a unique extension to an
operator in $\calL(E^{s}_{-\theta},E)$. We will denote this extension again by
$P(t,s)$. Below we will need to integrate $P(t,s)\phi(s)$ with respect to
$ds$. This can be made rigorous in the same way as in \cite{MSchn} using the
extension of $P(t,s)$ to $\calL(E^s_{-\theta},E)$. If $\phi$ is as above and
$(-A_{w}(\cdot))^{-\theta}\phi\in L^p(0,T;E)$ one could equivalently say $\phi
\in X_{-\theta}$ a.s., where $X = L^p(0,T;E)$ and $X_{-\theta}$ is the
extrapolation space under $A_w(\cdot)$ as a sectorial operator on
$L^p(0,T;E)$. Below we will not explicitly use the extrapolation spaces and
just interpret $P(t,s)\phi(s)$ as $P(t,s)(-A_{w}(s))^{\theta}
(-A_{w}(s))^{-\theta}\phi(s)$. This is allowed since for $x\in E^s_{-\theta}$
it is easily checked that
\[P(t,s) x = P(t,s)(-A_{w}(s))^{\theta} (-A_{w}(s))^{-\theta} x.\]

\begin{proposition}\label{prop:detconv}
Assume that (AT) and \ref{as:isom1a}  hold. Let $\theta\in [0,\mu)$
Let $p\in(1, \infty]$, $\delta\in [0,1)$ and $\lambda\in (0,1)$ be such that
$\lambda+\delta+\frac{1}{p}<\min\{1-\theta,\eta_0\}$. Then there exists a
constant $C_T$ with $\lim_{T\downarrow 0} C_T = 0$ such that for all
$(-A_w)^{-\theta}\phi\in L^p(0,T;E)$,
\begin{equation}\label{eq:detconv}
\|P*\phi\|_{C^\lambda([0,T];\tE_{\delta})} \leq C_T
\|(-A_w)^{-\theta}\phi\|_{L^p(0,T;E)}.
\end{equation}
\end{proposition}

\begin{proof}
First note that
\[\{(t,s):0\leq s<t\leq T\}\ni (t,s)\mapsto P(t,s) \phi(s) = P(t,s) (-A_{w}(s))^{\theta}
(-A_{w}(s))^{-\theta} \phi(s)\] is a strongly measurable $E$-valued function.

Let $\alpha>0$ be such that
$\lambda+\delta+\frac{1}{p}<\alpha<\min\{1-\theta,\eta_0\}$. Define
$\zeta_{\alpha}:[0,T]\to E$ as
\[\zeta_{\alpha}(t) = \frac{1}{\Gamma(1-\alpha)}\int_0^t (t-s)^{-\alpha} P(t,s) \phi(s) \, ds.\]
Then by \eqref{eq:2_15}, for each $t\in [0,T]$,
\begin{align*}
\|\zeta_{\alpha}(t)\|&\leq \frac{1}{\Gamma(1-\alpha)} \int_0^t (t-s)^{-\alpha}
\|P(t,s) \phi(s)\| \, ds \\ & \lesssim \int_0^t (t-s)^{-\alpha-\theta}
\|(-A_w(s))^{-\theta}\phi(s)\| \, ds.
\end{align*}
Therefore, by Young's inequality
\begin{align*}
\|\zeta_{\alpha}\|_{L^p(0,T;E)}^p &\lesssim \int_0^T \Big|\int_0^t
(t-s)^{-\alpha-\theta} \|(-A_w(s))^{-\theta}\phi(s)\| \, ds\Big|^{p} \, dt
\\ & \leq C_T^p \|(-A_w(s))^{-\theta}\phi\|_{L^p(0,T;E)}^p.
\end{align*}

Define $\zeta:[0,T]\to E$ as $\zeta = P*\phi$. By H\"olders's inequality and
$\theta<1-\frac{1}{p}$ this is well-defined. We claim that $\zeta =
R_{\alpha}(\zeta_{\alpha})$. This would complete the proof by Lemma \ref{fact2}
and
\begin{align*}
\|\zeta\|_{C^\lambda([0,T];\tE_{\delta})} &=
\|R_{\alpha}(\zeta_{\alpha})\|_{C^\lambda([0,T];\tE_{\delta})} \\ & \lesssim
C_T \|\zeta_{\alpha}\|_{L^p(0,T;E)}\lesssim
C_T\|(-A_w)^{-\theta}\phi\|_{L^p(0,T;E)}.
\end{align*}
To prove the claim notice that by Fubini's theorem for all $t\in [0,T]$,
\begin{align*}
R_{\alpha} (\zeta_\alpha) &= \frac{1}{\Gamma(\alpha)}\int_0^t (t-s)^{\alpha-1}
P(t,s) \zeta_\alpha(s)\, ds
\\ & = \frac{1}{\Gamma(1-\alpha)\Gamma(\alpha)}\int_0^t \int_0^s (t-s)^{\alpha-1} (s-r)^{-\alpha}
P(t,r)\phi(r) \, dr \, ds
\\ & = \frac{1}{\Gamma(1-\alpha)\Gamma(\alpha)}\int_0^t \int_r^t (t-s)^{\alpha-1} (s-r)^{-\alpha}
P(t,r)\phi(r)  \, ds \, dr
\\ & = \int_0^t P(t,r)\phi(r)  \, dr = \zeta(t).
\end{align*}
\end{proof}

\section{Stochastic convolutions\label{sec:stochconv}}
Let $(\O,\F,\P)$ be a complete probability space with a filtration $(\F_t)_{t\in
[0,T]}$. Let $E$ be a Banach space and $H$ be a separable Hilbert space. Let
$W_H$ be a cylindrical Wiener process with respect to $(\F_t)_{t\in [0,T]}$.
We strengthen the hypothesis \ref{as:isom1a} from page \pageref{as:isom}.

\let\ALTERWERTA\theenumi
\let\ALTERWERTB\labelenumi
\def\theenumi{(H1)$_{\eta_0}$}
\def\labelenumi{(H1)$_{\eta_0}$}
\begin{enumerate}
\item \label{as:isom} There exists an $\eta_0\in (0,1]$ and a family
of spaces $(\tE_{\eta})_{\eta\in [0,\eta_0]}$ such that each $\tE_{\eta}$ is a
UMD spaces with type $2$,
\[\tE_{\eta_0} \hookrightarrow \tE_{\eta_1} \hookrightarrow \tE_{\eta_2} \hookrightarrow \tE_{0}=E, \ \ \ 0\leq \eta_2\leq \eta_1\leq \eta_0.\]
and for all $\eta\in [0,\eta_0]$
\[(E, D(A(t)))_{\eta, 2} \hookrightarrow \tE_{\eta}\hookrightarrow E\]
with uniform constants in $t\in[0,T]$.
\end{enumerate}
\let\theenumi\ALTERWERTA
\let\labelenumi\ALTERWERTB

The next result will be formulated for a family $\{\Phi(t,\omega):t\in [0,T],
\omega\in\O\}$ such that for all $t\in [0,T]$ and all $\omega\in \O$,
$\Phi(t,\omega)\in \calL(H,E_{-\theta}^t)$, where $(-A_w)^{-\theta} \Phi$ is an
$H$-strongly measurable and adapted process from $[0,T]\times\O$ into $\calL(H,E)$ and
$\theta\in [0,\frac12)$ is fixed. We denote the stochastic convolution by
\[P\diamond \Phi(t) := \int_0^t P(t,s) \Phi(s) \, d W_H(s),\]
where $\Phi$ is as above.
The following extends results from \cite{DPKZ,Brz2,Sei}.

\begin{theorem}\label{thm:facttype2}
Assume (AT) and \ref{as:isom}. Let $\theta\in [0,\mu\wedge \tfrac12)$. Let
$p\in (2, \infty)$ and let $\delta,\lambda>0$ be such that
$\delta+\lambda+\frac1p<\min\{\frac12-\theta, \eta_0\}$. Let
$(-A_w)^{-\theta}\Phi:[0,T]\times \O\to \g(H,E)$ be $H$-strongly measurable and
adapted such that $(-A_w)^{-\theta}\Phi\in L^p(0,T;\g(H,E))$ a.s. Then for all
$t\in [0,T]$, $s\mapsto P(t,s)\Phi(s)\in \g(H,E)$ is $H$-strongly measurable
and adapted, $P \diamond \Phi$ exists in $\tE_{\delta}$ and is
$\lambda$-H\"older continuous and
there exists a constant $C\geq 0$ independent of $\Phi$ such that
\begin{equation}\label{eq:Cldptype2}
\begin{aligned}
\E\|P \diamond \Phi \|^p_{C^\lambda([0,T];\tE_{\delta})} & \leq C \E
\|(-A_w)^{-\theta}\Phi\|_{L^p(0,T;\g(H,E))}^p.
\end{aligned}
\end{equation}
\end{theorem}

\begin{proof}
We claim that
\[\{(t,s): 0\leq s<t\leq T\}\ni (t,s)\mapsto P(t,s)\Phi(s)\in \g(H,E)\]
is $H$-strongly measurable and for all $t\in [0,T]$ and
\[(0,t)\ni s\mapsto P(t,s)\Phi(s)\in \g(H,E)\]
is $H$-strongly adapted. Indeed, this follows from the assumption and the
remarks before Proposition \ref{prop:detconv} as soon as we write
\[P(t,s)\Phi(s) = P(t,s)(w-A(s))^{\theta}(w-A(s))^{-\theta}\Phi(s).\]

Let $\delta$ and $\lambda$ be as in the theorem and let
$\alpha$ be such that
$\delta+\lambda+\frac1p<\alpha<\min\{\frac12-\theta,\eta_0\}$. Define
$\zeta_\alpha:[0,T]\times\O\to E$ as
\[\zeta_\alpha(t) = \frac{1}{\Gamma(1-\alpha)}\int_0^t (t-s)^{-\alpha} P(t,s)\Phi(s) \, d
W_H(s).\] Then by Proposition \ref{prop:stochinttype2}, \eqref{eq:2_15}, Young's inequality and \cite[Appendix]{NVW3} $\zeta_{\alpha}$ is well-defined in $L^p((0,T)\times\O;E)$ and jointly measurable and
moreover we have
\begin{align*}
\|\zeta_{\alpha}\|_{L^p((0,T)\times\O;E)} &\lesssim  \Big(\E \int_0^T
\Big(\int_0^t \|(t-s)^{-\alpha} P(t,s)\Phi(s)\|^2_{\g(H,E)} \,
ds\Big)^{\frac{p}{2}} \, dt \Big)^{\frac1p}
\\ & \lesssim \Big(\E \int_0^T \Big(\int_0^t
(t-s)^{-2\alpha-2\theta} \|(-A_w(s))^{-\theta}\Phi(s)\|^2_{\g(H,E)}
\, ds\Big)^{\frac{p}{2}} \, dt \Big)^{\frac1p}
\\ & \leq C
\Big(\E\|(-A_w)^{-\theta}\Phi\|_{L^p(0,T;\g(H,E))}^p\Big)^{\frac1p}.
\end{align*}
Here we used $\alpha<\frac12-\theta$. Let $\O_0$ with $P(\O_0)=1$ be
such that $\zeta_\alpha(\cdot, \omega)\in L^p(0,T;E)$ for all
$\omega\in \O_0$. We may apply Lemma \ref{fact2} to obtain that for
all $\omega\in \O_0$,
\[R_\alpha \zeta_\alpha(\cdot, \omega) \in C^\lambda([0,T];\tE_{\delta})\]
and
\begin{equation}\label{eq:estRzeta1type2}
\|R_\alpha \zeta_\alpha(\cdot, \omega)\|_{C^\lambda([0,T];\tE_{\delta})}
\lesssim C \|\zeta_\alpha(\cdot, \omega)\|_{L^p(0,T;E)}.
\end{equation}

Define $\zeta:[0,T]\times\O\to E$ as $\zeta = P\diamond \Phi$. Since
$\theta<\frac12-\frac1p$, one may check that this is well-defined. We claim
that for all $t\in [0,T]$, for almost all $\omega\in \O$, we have
\begin{equation}\label{eq:equalconttype2}
\zeta(t,\omega) = (R_\alpha \zeta_\alpha(\cdot, \omega))(t).
\end{equation}
It suffices to check that for all $t\in [0,T]$ and $x^* \in E^*$,
almost surely we have
\[\lb \zeta(t), x^*\rb = \frac{1}{\Gamma(\alpha)}\int_0^t (t-s)^{\alpha-1} \lb P(t,s) \zeta_\alpha(s), x^*\rb \, ds.\]
As in Proposition \ref{prop:detconv} this follows from the (stochastic) Fubini
theorem (see \cite{DPKZ}). Therefore, the above estimates imply
\eqref{eq:Cldptype2}.

\end{proof}

For Hilbert spaces $E$ we can prove a maximal regularity result in
the non-autonomous setting. The autonomous case has been considered
in \cite[Theorem 6.14]{DPZ}. Our proof below is different from
\cite{DPZ} even in the autonomous case. We briefly recalled some results on $H^\infty$-calculus in Section \ref{subsec:Hinfty}. Note that we use formulations for $A(t)$ instead of $-A(t)$.

Assume (AT1) and the following condition on the operators
$(A(t))_{t\in [0,T]}$.
\begin{enumerate}
\item[($H^\infty$)] There exists constant $w\in \R$, $C>0$ and $\sigma\in \pi/2,\pi)$ such
that for all $t\in [0,T]$, $A_w(t)$ has a bounded
$H^\infty(\Sigma_{\sigma})$-calculus on
$\Sigma_{\varphi}$ and
\[C:=\sup_{t\in [0,T]}\big(\{\|f(A_w(t))\|:\ \|f\|_{H^\infty(\Sigma_\sigma)}\le 1\}\big) <\infty.\]
\end{enumerate}
Condition ($H^\infty$) has also appeared in \cite{VeZi} (with $A(t)$ replaced by $-A(t)$).
In the autonomous ($H^\infty$) has been used in \cite{DevNWe} to obtain maximal
regularity for equations with additive noise in Banach spaces. This has been
extended to the non-autonomous setting in \cite{VeZi}.

We reformulate the sufficient conditions from Remark \ref{rem:Hinfty0} in our situation here.
\begin{remark}\label{rem:Hinfty}
Each of the following two conditions is sufficient for ($H^\infty$) for a Hilbert space $E$.
\begin{enumerate}
\item If (AT1) holds and there exists a $w\in \R$, such that each $A_w(t)$ generates an analytic contraction semigroup, then ($H^\infty$) holds.

\item If there exists a $w\in \R$ such that each $-A_w(t)$ is positive and
self-adjoint, then for all $\sigma\in (\pi/2,\pi)$, the condition ($H^\infty$) holds with
$C=1$.
\end{enumerate}
\end{remark}

In the following result we use ($H^\infty$) to obtain a maximal regularity result for the stochastic convolution. Recall that by \eqref{eq:2_15}, $\|(-A_w(t))^{\frac12} P(t,s)\|\leq C (t-s)^{-\frac12}$.

\begin{theorem}\label{thm:maxreg}
Let $E$ be a Hilbert space. Assume that $(A(t))_{t\in [0,T]}$
satisfies (AT) and ($H^\infty$). If $\Phi:[0,T]\times \O\to \g(H,E)$
is $H$-strongly measurable and adapted and $\Phi\in
L^2(0,T;\g(H,E))$ a.s., then $(-A_w(\cdot))^{\frac12}P\diamond
\Phi\in L^2(0,T;E)$ a.s. Moreover there is a constant $C$
independent of $\Phi$ such that
\begin{equation}\label{eq:maxreg}
\E\|t\mapsto (-A_w(t))^{\frac12} (P\diamond
\Phi)(t)\|_{L^2(0,T;E)}^2 \leq C \E\|\Phi\|_{L^2(0,T;\g(H,E))}^2.
\end{equation}
\end{theorem}

For second order partial differential equations the inequality
\eqref{eq:maxreg} will allow us to derive $W^{1,2}(S)$-regularity,
where $W^{1,2}(S)$ denotes the Sobolev space on a domain $S\subset
\R^n$. Furthermore, \eqref{eq:maxreg} can be useful for
non-linear equations.

\begin{proof}


First assume that $\Phi\in L^2(\O;L^2(0,T;\g(H,E)))$. Notice that
$\g(H,E) = \mathcal{C}_2(H,E)$ is the space of Hilbert-Schmidt
operators from $H$ into $E$. Let $(h_n)_{n\geq 1}$ be an orthonormal
basis for $H$. By the It\^o isometry and the Fubini theorem, we have
\begin{align*}
\E\|t\mapsto (-A_w(t))^{\frac12}&(P\diamond
\Phi)(t)\|_{L^2(0,T;E)}^2
\\ &= \E \int_0^T \int_0^t \|(-A_w(t))^{\frac12} P(t,s)
\Phi(s)\|_{\g(H,E)}^2 \, ds \, dt \\ & = \E \int_0^T \sum_{n\geq
1}\int_s^T \|(-A_w(t))^{\frac12} P(t,s) \Phi(s) h_n \|^2 \, dt \,
ds.
\end{align*}

Let $P_w(t,s) = e^{w(t-s)}P(t,s)$. For $x\in E$ we can estimate
\begin{align*}
\Big(\int_s^T \|(-A_w(t))^{\frac12} P_w(t,s)
x\|^2 \, dt\Big)^{\frac12} \leq \sum_{i=1}^3 R_i.
\end{align*}
Here
\[
R_1^2 = \int_s^T \|(-A_w(t))^{\frac12} Z(t,s) x\|^2 \,
dt
\]
with $Z(t,s) = P_w(t,s) - \exp((t-s)A_w(t))$. It follows from
\cite[p.\ 144]{Ya90} and \cite[Lemma 3.2.1 and Theorem 3.2.2]{Am}
that
\[
\|(-A_w(t))^{\frac12} Z(t,s)\| \leq C_4
(t-s)^{-\frac12+\frac{\kappa_{\mu,\nu}}{2}}.
\]
Therefore, $R_1^2 \lesssim T^{\kappa_{\mu,\nu}} \|x\|$. Secondly, by
\cite[(2.4)]{Ya90}
\begin{align*}
R_2^2 &= \int_s^T \|(-A_w(t))^{\frac12} \exp((t-s) A_w(t)) x -
(-A_w(s))^{\frac12} \exp((t-s) A_w(s))x\|^2 \,dt
\\ & \lesssim \int_s^T (t-s)^{2\kappa_{\mu,\nu}-1}\, dt \|x\|
\lesssim T^{2\kappa_{\mu,\nu}}\|x\|.
\end{align*}
Finally, by ($H^\infty$) and \eqref{eq:squarefuncest0}
\[R_3^2 = \int_s^T \|(-A_w(s))^{\frac12} \exp((t-s) A_w(s)) x\|^2 \,
dt\lesssim \|x\|.\] It follows that

\begin{equation}\label{eq:resVx}
\Big(\int_s^T \|(-A_w(t))^{\frac12} P(t,s) x\|^2 \, dt\Big)^{\frac12}\lesssim
\|x\|.
\end{equation}
We may conclude that
\[\E\|(-A_w(\cdot))^{\frac12} P\diamond \Phi\|_{L^2(0,T;E)}^2\lesssim \E \int_0^T \sum_{n\geq 1} \|\Phi(s) h_n
\|^2 \, ds =  \E \|\Phi\|^2_{L^2(0,T;\g(H,E))}.
\]
This proves \eqref{eq:maxreg}.
The general result now follows from a localization argument.
\end{proof}

\section{The abstract evolution equation and solution concepts\label{sec:abstracteq}}
In this section and Section \ref{sec:Lipcoefinttype} let $E$, $H$,
$(\O,\F,\P)$, $(\F_t)_{t\in[0,T]}$ and $W_H$ be as in Section
\ref{sec:stochconv}. On the Banach space $E$ we consider the problem
\begin{equation}\tag{SE}\label{SEtype}
\left\{\begin{aligned}
dU(t) & = (A(t)U(t) + F(t,U(t)))\,dt + B(t,U(t))\,dW_H(t), \qquad t\in [0,T],\\
 U(0) & = u_0.
\end{aligned}
\right.
\end{equation}
Here $(A(t))_{t\in [0,T]}$ is a family of closed unbounded operators on $E$.
The initial value is a strongly $\F_0$-measurable mapping $u_0:\O\to E$.




We assume (AT) and \ref{as:isom}.  We assume $a\in [0,\eta_0)$ and for each $(t,\omega)\in [0,T]\times \O$, we assume that $F$ and $B$ map as follows
\[x\mapsto F(t,\omega,x) \ \text{maps from $\tE_{a}^t$ into $E_{-1}^t$},\]
\[x \mapsto B(t,\omega,x) \ \text{maps from $\tE_{a}^t$ into $\g(H,E_{-1}^t)$}.\]
More precisely, we have the following hypothesis on $F$ and $B$.

\let\ALTERWERTA\theenumi
\let\ALTERWERTB\labelenumi
\def\theenumi{(H2)}
\def\labelenumi{(H2)}
\begin{enumerate}
\item \label{as:LipschitzFtype} Let $a\in [0,\eta_0)$ and $\theta_F\in [0,\mu)$
be such that $a+\theta_F <1$. For all $x\in \tE_a$, $(t,
\omega)\mapsto (-A_w(t))^{-\theta_F} F(t, \omega,x)\in E$ is
strongly measurable and adapted. The function $(-A_w(t))^{-\theta_F}
F$ has linear growth and is Lipschitz continuous in space uniformly
in $[0,T]\times\O$, that is there are constants $L_F$ and $C_F$ such
that for all $t\in [0,T], \omega\in \O, x,y\in \tE_{a}$,
\begin{eqnarray}\label{eq:LipschitzF}
\|(-A_w(t))^{-\theta_F} (F(t,\omega,x)-F(t,\omega,y))\|_{E}&\leq &
L_F\|x-y\|_{\tE_{a}},
\\ \label{eq:linF}
\|(-A_w(t))^{-\theta_F} F(t,\omega,x)\|_{E}&\leq & C_F(1+\|x\|_{\tE_{a}}).
\end{eqnarray}
\end{enumerate}
\let\theenumi\ALTERWERTA
\let\labelenumi\ALTERWERTB

\let\ALTERWERTA\theenumi
\let\ALTERWERTB\labelenumi
\def\theenumi{(H3)}
\def\labelenumi{(H3)}
\begin{enumerate}
\item \label{as:LipschitzBtype} Let $a\in [0,\eta_0)$ and $\theta_B\in [0,\mu)$
be such that $a+\theta_B <\frac12$. For all $x\in \tE_a$, $(t, \omega)\mapsto
(-A_w(t))^{-\theta_B} B(t, \omega,x)\in \g(H,E)$ is strongly measurable and adapted.
The function $(-A_w)^{-\theta_B}B$ has linear growth and is Lipschitz
continuous in space uniformly in $[0,T]\times\O$, that is there are constants
$L_B$ and $C_B$ such that for all $t\in [0,T], \omega\in \O, x,y\in \tE_{a}$,
\begin{eqnarray}\label{eq:LipschitzB}
\|(-A_w(t))^{-\theta_B}(B(t,\omega,x)-B(t,\omega,y))\|_{\g(H,E)}&\leq &
L_B\|x-y\|_{\tE_{a}},
\\ \label{eq:linB}
\|(-A_w(t))^{-\theta_B}B(t,\omega,x)\|_{\g(H,E)}&\leq & C_B(1+\|x\|_{\tE_{a}}).
\end{eqnarray}
\end{enumerate}
\let\theenumi\ALTERWERTA
\let\labelenumi\ALTERWERTB

In our application in Section \ref{sec:SSV} we will not use
functions $F$ which take values in extrapolation spaces. However, in
forthcoming papers this will be important. In Section
\ref{sec:local} we will consider locally Lipschitz coefficients $F$
and $B$.



We introduce variational and mild solutions for \eqref{SEtype} and give
conditions under which both concepts are equivalent.

We need the adjoint operators $A(t)^*$. Note that these also satisfy (AT1). Since in our setting $E$ will be a UMD space with type $2$, it will also be reflexive. Therefore, Kato's result implies that also $A(t)^*$ is densely defined (cf.\ \cite[Section VIII.4]{Yos}).

For $t\in [0,T]$ let
\begin{align*}
\Gamma_t = \big\{ \varphi \in C^1([0,t];E^*) \,:
&\text{ for all } s \in [0,t] \ \varphi(s)\in D(A(s)^*) \\ & \text{
and } s\mapsto A(s)^*\varphi(s) \in C([0,t];E^*) \big\}.
\end{align*}
Fix some $t\in [0,T]$ and $\varphi\in \Gamma_t$. Formally, applying the It\^o
formula to $\lb U(t),\varphi(t)\rb$ yields
\begin{equation}\label{varsol}
\begin{aligned}
\lb U(t), &\varphi(t)\rb - \lb u_0, \varphi(0)\rb \\ & =  \int_0^t
\lb U(s), \varphi'(s) \rb \, ds+ \int_0^t \lb U(s), A(s)^*
\varphi(s) \rb + \lb F(s,U(s)), \varphi(s) \rb \, ds \\ & \qquad  +
\int_0^t B(s,U(s))^{*} \varphi(s) \, d W_H(s).
\end{aligned}
\end{equation}

\begin{definition}\label{def:var}
Assume (AT), \ref{as:isom}, \ref{as:LipschitzFtype} and
\ref{as:LipschitzBtype}. An $\tE_a$-valued process $(U(t))_{t\in [0,T]}$ is
called a {\em variational solution} of \eqref{SEtype}, if
\renewcommand{\labelenumi}{(\roman{enumi})}
\renewcommand{\theenumi}{(\roman{enumi})}
\begin{enumerate}
\item\label{en:var1} $U$ is strongly measurable and adapted, and in $L^2(0,T;\tE_a)$ a.s.

\item\label{en:var4} for all $t\in [0,T]$ and all $\varphi\in \Gamma_t$, almost
surely, \eqref{varsol} holds.
\end{enumerate}
\renewcommand{\labelenumi}{(\arabic{enumi})}
\renewcommand{\theenumi}{(\arabic{enumi})}
\end{definition}

The integrand $B(s,U(s))^{*} \varphi(s)$ of the stochastic integral in
\eqref{varsol} should be read as
\[((-A_w(s))^{-\theta_B} B(s,U(s)))^{*} (-A_w(s)^*)^{\theta_B}\varphi(s).\]
It follows from \ref{as:LipschitzBtype} that $s\mapsto ((-A_w(s))^{-\theta_B}
B(s,U(s)))^{*}$ is strongly measurable and adapted and in
$L^2(0,T;\calL(E^*,H^*))$ a.s. Moreover,
\[s\mapsto (-A_w(s)^*)^{\theta_B}\varphi(s) = (-A_w(s)^*)^{-1+\theta_B} (-A_w(s)^*)
\varphi(s)\] is in $C([0,t];E^*)$ by the H\"older continuity of
$(-A_w(s))^{-1+\theta_B}$ (cf.\ \cite[(2.10) and (2.11)]{Schn}) and its adjoint
and the assumption on $\varphi$. The integrand $\lb F(s,U(s)), \varphi(s) \rb$
has to be interpreted in a similar way.


Next we define a mild solution.

\begin{definition}\label{def:mild}
Assume (AT), \ref{as:isom}, \ref{as:LipschitzFtype} and
\ref{as:LipschitzBtype}. Let $r\in (2,\infty)$ be such that
$\theta_F<1-\frac1r$ and $\theta_B<\frac12-\frac1r$. We call an $\tE_a$-valued
process $(U(t))_{t\in [0,T]}$ a {\em mild
solution} of \eqref{SEtype}, if
\renewcommand{\labelenumi}{(\roman{enumi})}
\renewcommand{\theenumi}{(\roman{enumi})}
\begin{enumerate}
\item $U$ is strongly measurable and adapted, and in $L^r(0,T;\tE_a)$ a.s.

\item for all $t\in [0,T]$, a.s.
\[U(t) = P(t,0) u_0 + P*F(\cdot,U)(t) + P\diamond B(\cdot,U)(t) \ \text{in $E$}.\]
\end{enumerate}
\renewcommand{\labelenumi}{(\arabic{enumi})}
\renewcommand{\theenumi}{(\arabic{enumi})}
\end{definition}
Recall that $P*F(\cdot,U)$ and $P\diamond B(\cdot,U)$ stand for the
convolution and stochastic convolution as defined in Sections
\ref{sec:detconv} and \ref{sec:stochconv} respectively.

The stochastic convolution is well-defined. This follows if we write
\[P(t,s) B(s,U(s)) = P(t,s)(-A_w(s))^{\theta_B} (-A_w(s))^{-\theta_B} B(s,U(s))\]
and therefore by \eqref{eq:2_15} and H\"older's inequality
\begin{align*}
\int_0^t \|P(t,s) B(s,U(s))\|^2 \, ds & \lesssim \int_0^t
(t-s)^{-2\theta_B} \|(-A_w(s))^{-\theta_B} B(s,U(s))\|^2 \, ds\\ &
\lesssim 1+\|U\|_{L^r(0,T;\tE_a)}^2.
\end{align*}
In the same way one can see that the deterministic convolution is well-defined.
If $\theta_F=\theta_B=0$, then one may also take $r=2$ in Definition
\ref{def:mild}.

To prove equivalences between variational and mild solutions, we need the
following condition.

\let\ALTERWERTA\theenumi
\let\ALTERWERTB\labelenumi
\def\theenumi{(W)}
\def\labelenumi{(W)}
\begin{enumerate}
\item\label{eq:condW} Assume that for all $t\in [0,T]$, there is a $\sigma(E^*,
E)$-sequentially dense subspace $\Upsilon_t$\label{p:Gammat} of $E^*$
such that for all $x^*\in \Upsilon_t$, we have $\varphi(s) := P(t,s)^*
x^*$ is in $C^1([0,t];E^*)$ and $\varphi(s)\in D(A(s)^*)$ for all
$s\in [0,t]$ and
\begin{equation}\label{varphiMinA}
 \frac{d}{ds} \varphi(s) = -A(s)^* \varphi(s).
\end{equation}
\end{enumerate}
\let\theenumi\ALTERWERTA
\let\labelenumi\ALTERWERTB

The condition \ref{eq:condW} was introduced in \cite{VeZi} in order to relate
different solution concepts in the case of \eqref{SEtype} with additive noise.

\begin{remark}\label{rem:condW}
If (AT) holds for both for $A(t)$ and its adjoint, then
\ref{eq:condW} is fulfilled with $\Upsilon_t = D((A(t)^*)^2)$. This follows
from \cite[Theorem 6.1]{AT2}) and \cite[p. 1176]{AT3}. If $E$ is
reflexive, by Kato's result \cite{Kato}, one may take $\Upsilon_t =
D(A(t)^*)$.
\end{remark}

\begin{proposition}\label{prop:varmild}
Assume (AT), \ref{as:isom}, \ref{as:LipschitzFtype},
\ref{as:LipschitzBtype} and \ref{eq:condW}. Let $r\in (2,\infty)$ be
such that $\theta_B<\frac12-\frac1r$ and $\theta_F<1-\frac1r$. Let
$U:[0,T]\times\O\to \tE_a$ be strongly measurable and adapted and such
that $U\in L^r(0,T;\tE_a)$ a.s. The following assertions are
equivalent:
\begin{enumerate}
\item $E$ is a mild solution of \eqref{SEtype}.

\item $U$ is a variational solution of \eqref{SEtype}.
\end{enumerate}
\end{proposition}
Condition \ref{eq:condW} is only needed in $(2)\Rightarrow (1)$. If
$\theta_F=\theta_B=0$, then one may also take $r=2$ in Proposition
\ref{prop:varmild}. For the proof of the above equivalence we refer to the
appendix.

\section{Existence, uniqueness and regularity\label{sec:Lipcoefinttype}}

%


Assume (AT) and \ref{as:isom}. For $a\in[0,\eta_0)$ and $r\in [1, \infty)$ let
$Z_{a}^r$ be the closed subspace of adapted processes in
$C([0,T];L^r(\O;\tE_{a}))$. Assume \ref{as:LipschitzFtype} and
\ref{as:LipschitzBtype}, where $a\in[0,\eta_0)$.

Define the fixed point operator $L:Z_a^r \to Z_a^r$ as
\[L (\phi)= t\mapsto P(t,0) u_0 + P*F(\cdot,\phi)(t) + P\diamond B(\cdot,\phi)(t).\]
In the next lemma we show that $L$ is well-defined and that it is a strict
contraction in $Z_{a}^r$ for a suitable equivalent norm. Recall that
$P*F(\cdot,\phi)$ and $P\diamond B(\cdot,\phi)$ stand for the convolution and
stochastic convolution as defined in Sections \ref{sec:detconv} and
\ref{sec:stochconv} respectively.

\begin{lemma}\label{lem:contrtype}
Assume (AT), \ref{as:isom}, \ref{as:LipschitzFtype} and
\ref{as:LipschitzBtype}. Let $r\in (2, \infty)$ be such that
$a+\frac1r<\min\{\frac12-\theta_B,1-\theta_F,\eta_0\}$ and let
$u_0\in L^r(\O,\F_0;E_{a}^0)$. Then the operator $L$ is well-defined
and there is an equivalent norm $\dbn\cdot\dbn$ on $Z_a^r$ such that
for all $\phi_1, \phi_2\in Z_a^r$,
\begin{equation}\label{eq:fixpointesttype}
\dbn L(\phi_1) - L(\phi_2)\dbn_{Z_a^r} \leq \frac12 \dbn \phi_1 -
\phi_2\dbn_{Z_a^r}.
\end{equation}
Moreover, there is a constant $C$ independent of $u_0$ such that for all
$\phi\in Z_a^r$,
\begin{equation}\label{eq:operatoresttype}
\dbn L(\phi)\dbn_{Z_a^r}\leq C(1+(\E\|u_0\|_{E_{a}^0}^r)^{\frac1r}) +
\frac12\dbn \phi\dbn_{Z_a^r}.
\end{equation}
\end{lemma}


\begin{proof}

{\em Initial value part} --

By \eqref{eq:2_13} we may estimate
\[\|P(t,0) u_0\|_{E_{a}^t} \leq C\|u_0\|_{E_{a}^0}.\]
This clearly implies
\begin{equation}\label{eq:inittype2}
\|t\mapsto P(t,0) u_0\|_{Z_a^r} \lesssim \|u_0\|_{L^r(\O;E_a^0)},
\end{equation}
where the path continuity of $P(t,0) u_0$ in $\tE_a$ follows from Lemma
\ref{lem:initialcons}.

{\em Deterministic convolution} --

(a): \ Let $(-A_w)^{-\theta_F}\phi\in L^\infty(0,T;L^r(\O;E))$. Recall from the
proof of Proposition \ref{prop:detconv} that $P* \phi = \zeta =
R_{\alpha}(\zeta_\alpha)$. It follows from \eqref{eq:2_13} that for all $t\in
[0,T]$,
\begin{equation}\label{eq:conva}
\begin{aligned}
\|P*\phi(t)\|_{L^r(\O;\tE_a)} &=
\|R_{\alpha}(\zeta_\alpha)(t)\|_{L^r(\O;\tE_a)}
\\ & \lesssim \int_0^t (t-s)^{\alpha-1-a} \|\zeta_\alpha(s)\|_{L^r(\O;E)} \,
ds.
\end{aligned}
\end{equation}
By \eqref{eq:2_15} we obtain that
\begin{align*}
\|\zeta_\alpha(s)\|_{L^r(\O;E)} &\lesssim \Big\|\int_0^s (s-u)^{-\alpha}
\|P(s,u) \phi(u)\| \, du\Big\|_{L^{r}(\O)}
\\ & \lesssim \Big\|\int_0^s (s-u)^{-\alpha-\theta_F} \|(-A_w(u))^{-\theta_F} \phi(u)\| \, du\Big\|_{L^{r}(\O)}
\\ & \leq \int_0^s (s-u)^{-\alpha-\theta_F} \|(-A_w(u))^{-\theta_F} \phi(u)\|_{L^r(\O;E)} \,
du.
\end{align*}
If we combine this with \eqref{eq:conva} we obtain that for all $t\in [0,T]$
{\small
\begin{equation}\label{eq:convha2}
\begin{aligned}
\|&P\diamond \phi(t)\|_{L^r(\O;\tE_a)} \\ & \lesssim \int_0^t
(t-s)^{\alpha-1-a} \int_0^s (s-u)^{-\alpha-\theta_F} \|(-A_w(u))^{-\theta_F}
\phi(u)\|_{L^r(\O;E)} \, du \, ds
\\ & \eqsim \int_0^t (t-u)^{-a-\theta_F} \|(-A_w(u))^{-\theta_F}
\phi(u)\|_{L^r(\O;E)} \, du,
\end{aligned}
\end{equation}}
where in the last step we used Fubini's theorem and $\int_0^1
s^{-\alpha-\theta_F} (1-s)^{\alpha-1-a} \, ds$ is finite. Note that $P\diamond
\phi\in Z_a^r$ follows from the fact that $P\diamond \phi$ is also
(H\"older)-continuous by Proposition \ref{prop:detconv}.

(b): \ Let $\phi_1, \phi_2\in Z_a^r$. Then by \ref{as:LipschitzFtype},
$(-A_w)^{-\theta_F} F(\cdot, \phi_1)$ and $(-A_w)^{-\theta_F} F(\cdot, \phi_2)$
are adapted and in $L^\infty(0,T;L^r(\O;E))$ and by (a), $P*F(\cdot,\phi_1)$
and $P*F(\cdot,\phi_2)$ define an element of $Z_{a}^r$ and
{\small
\begin{equation}\label{eq:dettype21b2}
\begin{aligned}
\|&P* F(\cdot, \phi_1)(t) - P*F(\cdot, \phi_2)(t)\|_{L^r(\O;\tE_{a})}
\\ &\lesssim \int_0^t (t-s)^{-a} \|(-A_w(s))^{-\theta_F} F(s,\phi_1(s))-(-A_w(s))^{-\theta_F} F(s, \phi_2(s))\|_{L^r(\O;E)} \, ds \\ &\leq L_F
\int_0^t (t-s)^{-a-\theta_F} \|\phi_1(s)-\phi_2(s)\|_{L^r(\O;\tE_a)} \, ds.
\end{aligned}
\end{equation}}

{\em Stochastic convolution} --

(a): \ Let $(-A_w)^{-\theta_B}\Phi\in L^\infty(0,T;L^r(\O;\g(H,E)))$ be
adapted. Recall from the proof of Theorem \ref{thm:facttype2} that $P\diamond
\Phi = \zeta = R_{\alpha}(\zeta_\alpha)$. It follows from \eqref{eq:2_13} that
for all $t\in [0,T]$,
\begin{equation}\label{eq:convtype2stocha}
\begin{aligned}
\|P\diamond \Phi(t)\|_{L^r(\O;\tE_a)} &=
\|R_{\alpha}(\zeta_\alpha)(t)\|_{L^r(\O;\tE_a)} \\ & \lesssim \int_0^t
(t-s)^{\alpha-1-a} \|\zeta_\alpha(s)\|_{L^r(\O;E)} \, ds.
\end{aligned}
\end{equation}
By Proposition \ref{prop:stochinttype2} and \eqref{eq:2_15} we obtain that
\begin{align*}
\|\zeta_\alpha(s)\|_{L^r(\O;E)}^2 &\lesssim \Big\|\int_0^s (s-u)^{-2\alpha}
\|P(s,u) \Phi(u)\|_{\g(H,E)}^2 \, du\Big\|_{L^{r/2}(\O)}
\\ & \lesssim \Big\|\int_0^s (s-u)^{-2\alpha-2\theta_B} \|(-A_w(u))^{-\theta_B} \Phi(u)\|_{\g(H,E)}^2 \, du\Big\|_{L^{r/2}(\O)}
\\ & \leq \int_0^s (s-u)^{-2\alpha-2\theta_B} \|(-A_w(u))^{-\theta_B} \Phi(u)\|_{L^r(\O;\g(H,E))}^2 \,
du.
\end{align*}
If we combine this with \eqref{eq:convtype2stocha} we obtain that for all $t\in
[0,T]$ {\small
\begin{equation}\label{eq:convtype2stocha2}
\begin{aligned}
\|&P\diamond \Phi(t)\|_{L^r(\O;\tE_a)} \\ & \lesssim \int_0^t
(t-s)^{\alpha-1-a} \Big(\int_0^s (s-u)^{-2\alpha-2\theta_B}
\|(-A_w(u))^{-\theta_B} \Phi(u)\|_{L^r(\O;\g(H,E))}^2 \, du\Big)^{\frac12} \,
ds.
\end{aligned}
\end{equation}}
Note that $P\diamond \Phi\in Z_a^r$ follows from the fact that $P\diamond \Phi$
is also (H\"older)-continuous by Theorem \ref{thm:facttype2}.

(b): \ Let $\phi_1, \phi_2\in Z_a^r$. Then $(-A_w)^{-\theta_B}B(\cdot,\phi_1)$
and $(-A_w)^{-\theta_B}B(\cdot,\phi_2)$ are adapted and in
$L^\infty(0,T;L^r(\O;\g(H,E)))$. Denote
\[\Delta(\phi_1,\phi_2)(u) = (-A_w(u))^{-\theta_B} (B(u,\phi_1(u))-B(u, \phi_2(u))).\]
By (a) we obtain that $P\diamond B(\cdot,\phi_1)$ and $P \diamond
B(\cdot,\phi_2)$ are in $Z_a^r$ and

{\small
\begin{equation}\label{eq:convtype2stochb}
\begin{aligned}
\|P\diamond & B(\cdot, \phi_1)(t) - P\diamond B(\cdot, \phi_2)(t)\|_{L^r(\O;\tE_a)}\\
& \lesssim \int_0^t (t-s)^{\alpha-1-a} \Big(\int_0^s (s-u)^{-2\alpha-2\theta_B}
\|\Delta(\phi_1,\phi_2)(u)\|_{L^r(\O;\g(H,E))}^2 \, du \Big)^{\frac12}\, ds
\\ & \leq L_B \int_0^t (t-s)^{\alpha-1-a} \Big(\int_0^s (s-u)^{-2\alpha-2\theta_B}
\|\phi_1(u)-\phi_2(u)\|_{L^r(\O;\tE_a)}^2 \, du \Big)^{\frac12} \, ds.
\end{aligned}
\end{equation}
}

{\em Conclusions} --

It follows from the above considerations that $L$ is well-defined. For $p\geq
0$ define an equivalent norm on ${Z_a^r}$ by
\[\dbn \phi \dbn_{Z_a^r} = \sup_{t\in [0,T]} e^{-p t} \|\phi(t)\|_{L^r(\O;\tE_a)}.\]
We obtain that for $\phi_1,\phi_2\in Z_a^r$, we have
\begin{align*}
\dbn L(\phi_1) -L(\phi_2)\dbn_{Z_a^r} \leq  R_1 + R_2,
\end{align*}
where
\begin{align*}
R_1 &= \|P\diamond  B(\cdot, \phi_1)(t) - P\diamond B(\cdot,
\phi_2)(t)\|_{L^r(\O;\tE_a)},
\\ R_2 & = \|P*F(\cdot, \phi_1)(t) - P*F(\cdot,\phi_2)(t)\|_{L^r(\O;\tE_{a})}.
\end{align*}
It follows from \eqref{eq:dettype21b2} that
\begin{align*}
R_1 & \lesssim \sup_{t\in [0,T]} e^{-pt} \int_0^t (t-s)^{-a}
\|\phi_1-\phi_2\|_{L^r(\O;\tE_a)}\, ds  \\& = \sup_{t\in [0,T]} \int_0^t
e^{-p(t-s)}(t-s)^{-a} \, e^{-ps} \|\phi_1(s)-\phi_2(s)\|_{L^r(\O;\tE_a)}\, ds
\\ & \leq \int_0^T e^{-p s} s^{-a-\theta_F} \, ds \, \|\phi_1-\phi_2\|_{Z_{a}^r} = f(p,T,a,\theta_F) \|\phi_1-\phi_2\|_{Z_{a}^r},
\end{align*}
where $f(p,T,a,\theta_F)\downarrow 0$ as $p\to \infty$. Similarly, by
\eqref{eq:convtype2stochb}

{\small
\begin{align*}
R_2 & \lesssim \sup_{t\in [0,T]} e^{-pt} \int_0^t
(t-s)^{\alpha-1-a} \Big(\int_0^s (s-u)^{-2\alpha-2\theta_B}
\|\phi_1(u)-\phi_2(u)\|_{L^r(\O;\tE_a)}^2 \, du\Big)^{\frac12} \, ds
\\ & \leq \int_0^T e^{-ps} s^{\alpha-1-a} \, ds\, \Big(\int_0^T
e^{-2pu} u^{-2\alpha-2\theta_B} \, du\Big)^{\frac12} \,
\|\phi_1-\phi_2\|_{Z_a^r}
\\ & = g(p,T,a,\alpha,\theta_B) \|\phi_1-\phi_2\|_{Z_{a}^r},
\end{align*}}
where $g(p,T,a,\alpha,\theta_B)\downarrow 0$ as $p\to \infty$.

Taking $p$ large gives \eqref{eq:fixpointesttype}.  Moreover, the estimate
\eqref{eq:operatoresttype} follows from \eqref{eq:fixpointesttype} and
\[\dbn L(0)\dbn_{Z_a^r}\leq C(1+\|u_0\|_{L^r(\O;E_a^0)}).\]
\end{proof}

We can now obtain a first existence, uniqueness and regularity result for
\eqref{SEtype}.

\begin{proposition}\label{prop:mainexistenceLtype}
Assume (AT1), (AT2), \ref{as:isom}, \ref{as:LipschitzFtype} and
\ref{as:LipschitzBtype}. Let $r\in (2, \infty)$ be such that
$a+\frac1r<\min\{\frac12-\theta_B,1-\theta_F,\eta_0\}$ and let $u_0\in
L^r(\O,\F_0;E_{a}^0)$. Then the following assertions hold:
\begin{enumerate}
\item There exists a unique mild solution $U\in Z_a^r$ of \eqref{SEtype}.
Moreover, there exists a constant $C\geq 0$ independent of $u_0$ such that
\begin{equation}\label{eq:mainthmestimateLtype}
\|U\|_{Z_a^r} \leq C ( 1 + (\E\|u_0\|^r_{E_{a}^0})^{\frac{1}{r}}).
\end{equation}
\item There exists a version of $U$ in $L^r(\O;C([0,T];\tE_a))$. Furthermore,
for every $\delta,\lambda>0$ such that $\delta+a+\lambda+\frac1r
<\min\{\frac12-\theta_B,1-\theta_F,\eta_0\}$ there exists a version of $U$ such
that $U-P(\cdot,0)u_0$ in $L^r(\O;C^\lambda([0,T];\tE_{\delta+a}))$ and there
is a constant $C$ independent of $u_0$ such that
\begin{equation}\label{eq:mainthmestimateholdertype}
\Big(\E\|(U-P(\cdot,0)u_0)\|_{C^\lambda([0,T];\tE_{\delta+a})}^r\Big)^{\frac{1}{r}}
\leq C ( 1 + (\E\|u_0\|^r_{E_a^0})^{\frac{1}{r}}.
\end{equation}
\end{enumerate}
\end{proposition}

If $u_0\in L^r(\O;E_{\delta+a+\lambda}^0)$, then the same regularity
as in \eqref{eq:mainthmestimateholdertype} can be derived for the
solution $U$. Indeed, by Lemma \ref{lem:initialcons}
$P(\cdot,0)u_0\in L^r(\O;C^\lambda([0,T];\tE_{\delta+a}))$.

\begin{proof}
(1): \ It follows from Lemma \ref{lem:contrtype} that there exists a unique fix
point $U\in Z_a^r$ of $L$. It is clear from the definition of $L$ that $U$ is
the unique mild solution in $Z_a^r$.

(2): \ By Proposition \ref{prop:detconv} we obtain that
\[
\E\|P*F(\cdot,U)\|_{C^{\lambda}([0,T];\tE_{a+\delta})}^r \lesssim \E\|
(-A_w)^{-\theta_F} F(\cdot,U)\|_{L^r(0,T;E)}^r \lesssim 1 + \|U\|_{Z_a^r}.
\]
It follows from Theorem \ref{thm:facttype2} that
\begin{align*}
\E\|P\diamond B(\cdot,U)\|_{C^{\lambda}([0,T]; \tE_{a+\delta})}^r & \lesssim
\E\|(-A_w)^{-\theta_B}B(\cdot, U(s))\|_{L^r(0,T;\g(H,E))}^r \lesssim 1 +
\|U\|_{Z_a^r}.
\end{align*}

Define $\tilde U:[0,T]\times\O\to \tE_{a}$ as
\[\tilde U(t) = P(t,0) u_0 + P*F(\cdot,{U})(t)+P\diamond B(\cdot,{U})(t),\]
where we take the versions of the convolutions as above. Clearly, $\tilde U=U$
in $Z_a^r$ and therefore $\tilde U$ is the required mild solution. Moreover
there is a constant $C$ such that
\[\E\|\tilde U-P(\cdot,0)u_0)\|_{C^{\lambda}([0,T];
\tE_{a+\delta})}^r\leq C(1+\|\tilde U\|_{Z_a^r}).
\]
Now \eqref{eq:mainthmestimateholdertype} follows from
\eqref{eq:mainthmestimateLtype}.

\end{proof}

\begin{theorem}\label{thm:mainexistencegeninitial}
Assume (AT1), (AT2), \ref{as:isom}, \ref{as:LipschitzFtype} and
\ref{as:LipschitzBtype}. Let $u_0:\O\to E_{a}^0$ be strongly $\F_0$ measurable.
Then the following assertions hold:
\begin{enumerate}
\item There exists a unique mild solution $U$ of \eqref{SEtype} with paths in
$C([0,T];\tE_a)$ a.s.

\item For every $\delta,\lambda>0$ with $\delta+a+\lambda
<\min\{\frac12-\theta_B,1-\theta_F,\eta_0\}$ there exists a version
of $U$ such that $U-P(\cdot,0)u_0$ in
$C^\lambda([0,T];\tE_{\delta+a})$ a.s.
\end{enumerate}
\end{theorem}
As below Proposition \ref{prop:mainexistenceLtype} if $u_0\in
E^0_{\delta+a+\lambda}$ a.s, then $U$ has a version with paths in
$C^\lambda([0,T];\tE_{\delta+a})$ for $\delta$ and $\lambda$ as in
Theorem \ref{thm:mainexistencegeninitial} (2).

For the proof we need the following lemma.

\begin{lemma}\label{lem:localityexistencetype}
Under the conditions of Proposition \ref{prop:mainexistenceLtype}
let $U$ and $V$ in the space $L^r(\O;C([0,T];\tE_a))$ be the mild
solutions of \eqref{SEtype} with initial values $u_0$ and $v_0$ in
$L^r(\O,\F_0;E_a^0)$. Then almost surely on the set $\{u_0 = v_0\}$
we have $U \equiv V$.
\end{lemma}
\begin{proof}
Let $\Gamma = \{u_0 = v_0\}$. Since $\Gamma$ is $\F_0$-measurable it follows
from Lemma \ref{lem:contrtype} that
\begin{align*}
\dbn U\one_{\Gamma}-V\one_{\Gamma}\dbn_{Z_r^a} &= \dbn
L(U)\one_{\Gamma}-L(V)\one_{\Gamma}\dbn_{Z_r^a}
\\ & = \dbn(L(U\one_{\Gamma})-L(V\one_{\Gamma}))\one_{\Gamma}\dbn_{Z_r^a}
\\ & \leq \frac12
\dbn U\one_{\Gamma}-V\one_{\Gamma}\dbn_{Z_r^a},
\end{align*}
hence $U|_{[0,T]\times\Gamma}= V|_{[0,T]\times\Gamma}$ in $Z_a^r$. The result
now follows from the path continuity of $U$ and $V$.
\end{proof}

\begin{proof}[Proof of Theorem \ref{thm:mainexistencegeninitial}]
Let $r>2$ be such that $\delta+a+\lambda
+\frac1r<\min\{\frac12-\theta_B,1-\theta_F,\eta_0\}$. Define
$(u_n)_{n\geq 1}$ in $L^r(\O,\F_0;E_a^0)$ as $u_n =
\one_{\{\|u_0\|\leq n\}} u_0$. By Proposition
\ref{prop:mainexistenceLtype}, for each $n\geq 1$, there is a mild
solution $U_n\in Z_a^r$ of \eqref{SEtype} with initial value $u_n$
and we may take the version of $U_n$ from Proposition
\ref{prop:mainexistenceLtype} (2). Lemma
\ref{lem:localityexistencetype} implies that for $1\leq m\leq n$
almost surely on the set $\{\|u_0\|\leq m\}$, for all $t\in [0,T]$,
$U_n(t) = U_m(t)$. It follows that almost surely, for all $t\in
[0,T]$, $\limn U_n(t)$ exists in $\tE_a$. Define $U:[0,T]\times\O\to
\tE_a$ as $U(t) = \limn U_n(t)$ if this limit exists and $0$
otherwise. Clearly, $U$ is strongly measurable and adapted.
Moreover, almost surely on $\{\|u_0\|\leq n\}$, for all $t\in
[0,T]$, $U(t) = U_n(t)$ and hence $U-P(\cdot,0) u_0$ has the same
regularity as $U_n-P(\cdot,0) u_n$. It can be easily checked that
$U$ is a mild of \eqref{SEtype} satisfying (2).

{\em Uniqueness:} Let $U^1,U^2\in C([0,T];\tE_a)$ a.s. be mild solutions of
\eqref{SEtype}. For each $n\geq 1$ and $i=1,2$ define the stopping times
$\nu_n^i$ as
\[\nu_n^i = \inf\Big\{t\in[0,T]:\|U^i(t)\|_{\tE_{a}} \geq n\Big\}.\]
For each $n\geq 1$ let $\tau_n = \nu_n^1 \wedge \nu_n^2$, and let $U^1_n =
U^1\one_{[0,\tau_n]}$ and $U^2_n = U^2\one_{[0,\tau_n]}$. Then for all $n\geq
1$, $U^1_n$ and $U^2_n$ are in $L^r(\O;L^\infty(0,T;\tE_a))$ for all $r<\infty$
so in particular in $L^\infty(0,T;L^r(\O;\tE_a))$ for all $r<\infty$. One
easily checks that
\[
U^i_n = \one_{[0,\tau_n]} (L(U^i_n))^{\tau_n}, \ i=1,2,
\]
where $L$ is the mapping introduced before Lemma \ref{lem:contrtype} and
\[(L (U^i_n))^{\tau_n}(t) := (L (U^i_n))(t\wedge \tau_n),\ i=1,2.\]
One can check that Lemma \ref{lem:contrtype} remains valid if $Z_a^r$ is
replaced by $\widehat{Z}_a^r$ the space of all adapted processes in
$L^\infty(0,T;L^r(\O;\tE_a))$. Therefore,
\begin{align*}
\dbn U^1_n - U^2_n\dbn_{\widehat{Z}_a^r}  &= \dbn \one_{[0,\tau_n]}(
L(U^1_n)^{\tau_n} - L(U^2_n)^{\tau_n})\dbn_{\widehat{Z}_a^r}
\\ & \leq \dbn L(U^1_n) - L(U^2_n)\dbn_{\widehat{Z}_a^r}
\\ & \leq \frac12 \dbn U^1_n - U^2_n\dbn_{\widehat{Z}_a^r}.
\end{align*}
This implies that $U^1_n = U^2_n$ in $\widehat{Z}_a^r$. Therefore, for all
$t\leq \tau_n$, $U^1(t) = U^2(t)$ a.s.\ Letting $n$ tend to infinity yields
that for all $t\in [0,T]$, $U^1(t) = U^2(t)$ a.s.\ and by path-continuity this
implies that a.s. for all $t\in [0,T]$, $U^1(t) = U^2(t)$.
\end{proof}

\section{Local mild solutions\label{sec:local}}

Next we extend the results to the case where $F$ and $B$ are locally
Lipschitz. This is a standard procedure (cf.\ \cite{Brz2,NVW3,Sei} and
references therein), but we believe it is better to include it here
for completeness. Assume (AT) and \ref{as:isom}.

\let\ALTERWERTA\theenumi
\let\ALTERWERTB\labelenumi
\def\theenumi{(H2)$'$}
\def\labelenumi{(H2)$'$}
\begin{enumerate}
\item \label{as:LipschitzFtypeloc} Let $a\in [0,\eta_0)$ and $\theta_F\in
[0,\mu)$ be such that $a+\theta_F <1$. For all $x\in \tE_a$, $(t,
\omega)\mapsto (-A_w(t))^{-\theta_F}  F(t, \omega,x)\in E$ is
strongly measurable and adapted. The function $(-A_w(t))^{-\theta_F}
F$ is locally Lipschitz continuous in space uniformly in
$[0,T]\times\O$, that is for each $R>0$ there is a constant
$L_{F,R}$ such that for all $t\in [0,T], \omega\in \O, x,y\in
\tE_{a}$ with $\|x\|_{\tE_a},\|y\|_{\tE_a}\leq R$,
\begin{eqnarray}\nonumber
\|(-A_w(t))^{-\theta_F} (F(t,\omega,x)-F(t,\omega,y))\|_{E}&\leq &
L_{F,R}\|x-y\|_{\tE_{a}}.
\end{eqnarray}
\end{enumerate}
\let\theenumi\ALTERWERTA
\let\labelenumi\ALTERWERTB

\let\ALTERWERTA\theenumi
\let\ALTERWERTB\labelenumi
\def\theenumi{(H3)$'$}
\def\labelenumi{(H3)$'$}
\begin{enumerate}
\item \label{as:LipschitzBtypeloc} Let $a\in [0,\eta_0)$ and $\theta_B\in
[0,\mu)$ be such that $a+\theta_B <\frac12$. For all $x\in \tE_a$,
$(t, \omega)\mapsto (-A_w(t))^{-\theta_B} B(t, \omega,x)\in E$ is
strongly measurable and adapted. The function $(-A_w)^{-\theta_B}B$
has linear growth and is locally Lipschitz continuous in space
uniformly in $[0,T]\times\O$, that is for each $R>0$ there is a
constant $L_{B,R}$ such that for all $t\in [0,T], \omega\in \O,
x,y\in \tE_{a}$ with $\|x\|_{\tE_a},\|y\|_{\tE_a}\leq R$,
\begin{eqnarray}
\nonumber
\|(-A_w(t))^{-\theta_B}(B(t,\omega,x)-B(t,\omega,y))\|_{\g(H,E)}&\leq
& L_{B,R}\|x-y\|_{\tE_{a}},
\end{eqnarray}
\end{enumerate}
\let\theenumi\ALTERWERTA
\let\labelenumi\ALTERWERTB

We recall the definition of an admissible
process and a local mild solution.
Let $T>0$ and let $\tau$ be a stopping time with values in $[0,T]$.
For $t\in [0,T]$ let
\[\O_t(\tau) = \{\omega\in \O: t<\tau(\omega)\},\]
\[[0,\tau)\times\O=\{(t,\omega)\in [0,T]\times\O: 0\leq t<\tau(\omega)\}.\]
A process $\zeta:[0,\tau)\times \O \to E$ (or $(\zeta(t))_{t\in
[0,\tau)})$ is called {\em admissible} if for all $t\in [0,T]$,
$\O_t(\tau)\ni\omega\to \zeta(t,\omega)$ is $\F_t$-measurable and
for almost all $\omega\in \O$, $[0,\tau(\omega))\ni t\mapsto
\zeta(t,\omega)$ is continuous.

\begin{definition}
Assume (AT), \ref{as:isom}, \ref{as:LipschitzFtypeloc} and \ref{as:LipschitzBtypeloc}. We call an admissible $\tE_a$-valued process $(U(t))_{t\in
[0,\tau)}$ a {\em local mild solution} of \eqref{SEtype}, if
$\tau\in(0,T]$,
$\tau = \limn \tau_n$, where
\begin{equation}\label{eq:taunU}
\tau_n = \inf\{t\in [0,T]: \|U(t)\|_{\tE_a}\geq n\}, \ \ n\geq 1
\end{equation}
and such that for all $t\in [0,T]$ and all $n\geq 1$, the
following condition holds:
for all $t\in [0,T]$, a.s.
\begin{align*}
U(t\wedge \tau_n) = P(t\wedge \tau_n,0) u_0 & + \int_0^{t\wedge \tau_n} P(t\wedge \tau_n,s)F(s,U(s\wedge \tau)) \one_{[0,\tau_n]}(s) \, ds \\ & + I_{\tau_n}(B(\cdot,U))(t\wedge \tau_n).
\end{align*}
\end{definition}
In \eqref{eq:taunU} we take $\tau_n = T$ if the infimum is taken over the empty set.
By \ref{as:LipschitzFtypeloc} and Proposition \ref{prop:detconv} the deterministic convolution is well-defined and pathwise continuous. The process $I_{\tau_n}(B(\cdot,U))$ is defined by
\[I_{\tau_n}(B(\cdot,U)) (t) = \int_0^t P(t,s) B(s,U(s\wedge \tau_n)) \one_{[0,\tau_n]}(s) \, d W_H(s).\]
This process is well-defined and pathwise continuous by Theorem \ref{thm:facttype2}. Therefore, $I_{\tau_n}(B(\cdot,U))(t\wedge \tau_n)$ is well-defined.
The motivation for defining $I_{\tau_n}$ in this way is explained in the appendix of \cite{BMS}.
It is needed in order to avoid the use of the process
\[s\mapsto P(t\wedge \tau_n,s) B(s,U(s\wedge \tau_n))\one_{[0,\tau_n]}(s),\]
which is not adapted, since $P(t\wedge \tau_n,s)x$ is not adapted for $x\in E\setminus \{0\}$. This problem seems to be overlooked in some of the existing literature, and the referee kindly communicated the problem and \cite{BMS} to the author.

%

\medskip

For $a\in[0,\eta_0)$ and $r\in [1,
\infty)$ let $Z_{a,\text{adm}}(\tau)$ be the space of
$\tE_a$-valued admissible processes $(\phi(t))_{t\in [0,\tau)}$.
A local mild solution $(U(t))_{t\in [0,\tau)}$ is called {\em
maximal} for the space $Z_{a,\text{adm}}(\tau)$ if for any other local mild solution
$(\tilde{U}(t))_{t\in[0,\tilde{\tau})}$ in $Z_{a,\text{adm}}(\tau)$, almost surely we
have $\tilde \tau\leq \tau$ and $\tilde{U}\equiv U|_{[0,\tilde
\tau)}$. Clearly, a maximal local mild solution
is always unique in $Z_{a,\text{adm}}(\tau)$. We say that a local mild solution
$(U(t))_{t\in [0,\tau)}$ of \eqref{SEtype} is a {\em global mild
solution} of \eqref{SEtype} if $\tau = T$ almost surely and $U$ has
an extension to a mild solution $\hat{U}:[0,T]\times\O\to \tE_a$ of
\eqref{SEtype}. In particular, almost surely ``no blow" up occurs at
$t=T$.

We say that $\tau$ is an {\em explosion time} if for almost all
$\omega\in \O$ with $\tau(\omega)<T$,
\[
\limsup_{t\uparrow \tau(\omega)} \|U(t,\omega)\|_{\tE_a} = \infty.
\]
Notice that if $\tau=T$ almost surely, then $\tau$ is always an
explosion time in this definition. However, there does not have to
be any ``blow up" in this case.

\begin{theorem}\label{thm:mainexistenceLocallyLipschitzCasetype}
Assume (AT1), (AT2), \ref{as:isom}, \ref{as:LipschitzFtypeloc} and
\ref{as:LipschitzBtypeloc}. Let $u_0:\O\to E_{a}^0$ be strongly
$\F_0$-measurable. Then the following assertion hold:
\begin{enumerate}
\item There exists a unique maximal local mild solution
$(U(t))_{[0,\tau)}$ in $Z_{a,\text{adm}}(\tau)$ of \eqref{SEtype}.

\item For every $\lambda,\delta>0$ with
$\lambda+\delta+a<\min\{1-\theta_F, \frac{1}{2}-\theta_B,\eta_0\}$
there exists a version of $U$ such that for almost all $\omega\in
\O$,
\[t\mapsto U(t,\omega) - P(t,0) u_0(\omega)\in C^\lambda_{loc}([0,\tau(\omega));\tE_{a+\delta}).\]
\end{enumerate}

If, additionally $F$ and $B$ are of linear growth, i.e.\
\eqref{eq:linF} and \eqref{eq:linB} hold, then the following
assertions hold:
\begin{enumerate}
\item[(3)]\label{it:lingrowth1type} The function
$U$ from (1) and (2) is the unique global mild solution of
\eqref{SEtype} with paths in $C([0,T];\tE_a)$ and the statements of
Theorem \ref{thm:mainexistencegeninitial} hold.

\item[(4)]\label{it:lingrowth2type}
If $r\in (2, \infty)$ is such that $a+\frac1r<\min\{1-\theta_F,
\frac{1}{2}-\theta_B,\eta_0\}$ and $u_0\in L^r(\O,\F_0;E^0_a)$, then
the solution $U$ is in $Z_{a}^r$ and the statement of Proposition
\ref{prop:mainexistenceLtype} hold.
\end{enumerate}
\end{theorem}

The proof is based on the following local uniqueness result.
\begin{lemma}\label{lem:localuniqueness2type}
Assume that the conditions of Theorem
\ref{thm:mainexistenceLocallyLipschitzCasetype} are satisfied.
Assume that $(U_1(t))_{t\in [0,\tau_1)}$ in
$Z_{a,\text{adm}}(\tau_1)$ and $(U_2(t))_{t\in [0,\tau_2)}$ in
$Z_{a,\text{adm}}(\tau_2)$ are local mild solution of \eqref{SEtype}
with initial values $u_0^1$ and $u_0^2$. Let $\Gamma =
\{u_0^1=u_0^2\}$. Then almost surely on $\Gamma$,
$U_1|_{[0,\tau_1\wedge \tau_2)} \equiv U_2|_{[0,\tau_1\wedge
\tau_2)}$. Moreover, if $\tau_1$ is an explosion time for $U_1$,
then almost surely on $\Gamma$, $\tau_1 \geq \tau_2$. If $\tau_1$
and $\tau_2$ are explosion times for $U_1$ and $U_2$, then almost
surely on $\Gamma$, $\tau_1=\tau_2$ and $U_1\equiv U_2$.
\end{lemma}

Both results can be proved using standard localization techniques.
We refer the reader to \cite[Section 4]{Brz2}, \cite[Section
5]{Sei}, \cite[Section 8]{NVW3} or \cite[Chapter 8]{VThesis} for a
proof in a framework close to the one above.

\section{Examples\label{sec:SSV}}

Below we consider the stochastic partial differential equation from
\cite{SSV}. We will apply Theorem \ref{thm:mainexistencegeninitial}
and Theorem \ref{thm:mainexistenceLocallyLipschitzCasetype} to
obtain existence, uniqueness and regularity of mild solutions. By
Proposition \ref{prop:varmild} this will also give the unique
variational solution. The operator $A(t)$ will be a time dependent
second order elliptic differential operator with (time-dependent)
Neumann boundary conditions. As in \cite{SSV} we consider second
order equations with noise that is white with respect to the time
variable and colored with respect to space variable. We will reprove
and improve some of the regularity results from \cite{SSV} using the
results of the previous sections. This will be done in three
examples below

Recall that $a$ is the parameter for the solution space $\tE_a$. For
the Examples \ref{ex:Lip1} and \ref{ex:Lip2} it will suffice to take
$a=0$ in Theorem \ref{thm:mainexistencegeninitial}. In Example
\ref{ex:loclLip} we consider the locally Lipschitz case, and there
we need $a>0$. The parameter $\theta_B$ we allow us to consider
covariance operators which are not necessarily of trace class. For
details on covariance operators we refer to \cite{Bog,DPZ}.

\begin{remark}\label{rem:othercases}
Some other examples which fit into our general framework:
\begin{enumerate}
\item Higher order equations, possibly driven by multiplicative space-time white noise.
Note that for second order equations, this is only possible for
dimension one, and therefore not very illustrative for our setting.
In regular bounded domains in $\R^n$ one can consider multiplicative
space-time white noise if the order of the elliptic operator $2m>n$
(see \cite{NVW3} for the autonomous case and \cite[Chapter
8]{VThesis} for the non-autonomous case with Dirichlet boundary
conditions).

\item $F$ and $B$ could be (non)-linear differential operators of lower
order.

\item Equations with boundary noise. This is work in progress
\cite{VSchn}.
\end{enumerate}
\end{remark}

Let us first recall some basic notations (cf.\ \cite{Tr1}). Let $S$
be a bounded domain and $m\in \N$, $p,q\in [1, \infty]$, $s\in \R$,
$\beta_1,\beta_2\in (0,1)$. $W^{m,p}(S)$ will be the Sobolev space.
$B^{s}_{p,q}(S)$ will be the Besov space. $H^{s,p}(S)$ is the Bessel
potential space and $H^{s}(S) := H^{2,p}(S)$ and $H^{m,p}(S) =
W^{m,p}(S)$. $C^{\delta}(\overline{S})$ is the space of
$\delta$-H\"older continuous functions.
$C^{\beta_1,\beta_2}(\overline{S}\times[0,T])$ is the space of
functions $f:S\times[0,T]\to \R$ which satisfy
\begin{align*}
|f(s_1,t) - f(s_2,t)| & \leq C_1|s_1-s_2|^{\beta_1}, \ s_1,s_2\in S,
\ t\in [0,T],
\\ |f(s,t_1) - f(s,t_2)| & \leq C_2|t_1-t_2|^{\beta_2}, \ s\in S, \
t_1,t_2\in [0,T]
\end{align*}
for certain constants $C_1,C_2\geq 0$. Clearly,
$C^{\beta_1,\beta_2}(\overline{S}\times[0,T])\hookrightarrow
C^{\beta_1\wedge \beta_2}(\overline{S}\times[0,T])$.

\begin{example}\label{ex:Lip1}
Let $(\O,\F,\P)$ be a complete probability space with a filtration $(\F_t)_{t\in
[0,T]}$. Consider
\begin{equation}\label{neumann}
\begin{aligned}
d u(t,s) &= A(t, s, D) u(t, s) + f(t,s, u(t,s)) \, dt \\ &\qquad \qquad
 \qquad + g(t,s, u(t,s)) \, d W(t,s), \ \ t\in (0,T], s\in S,
\\ C(t,s,D) u(t,s) &= 0, \ \ t\in (0,T], s\in \partial S
\\ u(0, s) &= u_0(s), \ \ s\in S.
\end{aligned}
\end{equation}
Here $S$ is a bounded domain with boundary of class $C^2$ and outer normal
vector $n(s)$ in $\R^n$, and
\begin{align*}
A(t,s,D) &= \sum_{i,j=1}^n D_i \Big( a_{ij}(t,s) D_j \Big) + a_0(t,s),
\\  C(t,s,D) &= \sum_{i,j=1}^n a_{ij}(t,s) n_i(s) D_j.
\end{align*}
We assume that the coefficients are real and satisfy
\begin{align*}
a_{ij} &\in C^{\mu}([0,T];C(\overline{S})), \ a_{ij}(t, \cdot)\in
C^1(\overline{S}), \ D_k a_{ij}\in C([0,T]\times\overline{S}),
\\ a_0&\in C^{\mu}([0,T],L^n(S))\cap C([0,T];C(\overline{S}))
\end{align*}
for $i, j, k=1, \ldots, n$, $t\in [0,T]$, and a constant
$\mu\in(\tfrac{1}{2},1]$. Furthermore, let $(a_{ij})$ be symmetric
and assume that there exists a $\kappa>0$ such that
\begin{equation}\label{unifell}
\sum_{i,j=1}^n a_{ij}(t,s) \xi_i \xi_j \ge \kappa |\xi|^2, \ \ s \in
\overline{S}, t \in [0,T], \xi \in \R^n.
\end{equation}

Let $f,g:[0,T]\times\O\times S\times\R\to \R$ be measurable, adapted and
Lipschitz functions with linear growth uniformly $\O\times[0,T]\times S$, i.e.\
there exist $L_f,C_f,L_g,C_g$ such that for all $t\in [0,T]$, $\omega\in\O$,
$s\in \R$ and $x,y\in\R$,
\begin{eqnarray}\label{eq:Lipf}
|f(t,\omega,s,x) - f(t,\omega,s,y)|&\leq &L_f |x-y|,
\\ \label{eq:linf}  |f(t,\omega,s,x)| & \leq & C_f
(1+|x|),
\\ \label{eq:Lipg}
|g(t,\omega,s,x) - g(t,\omega,s,y)| & \leq & L_g |x-y|,
\\ \label{eq:ling} |g(t,\omega,s,x)| & \leq & C_g
(1+|x|).
\end{eqnarray}

The noise term $W$ is an $L^2(S)$-valued Brownian motion with respect to
$(\F_t)_{t\in [0,T]}$. We assume that it has a covariance $Q\in \calL(L^2(S))$
which satisfies
\begin{equation}\label{eq:condCov}
\sqrt{Q}\in \calL(L^2(S),L^\infty(S)).
\end{equation}

The following statements hold:
\begin{enumerate}
\item Let $p\in [2,\infty)$. If $u_0\in L^p(S)$ a.s., then there exists a
unique mild and variational solution $u$ of \eqref{neumann} with paths in
$C([0,T];L^p(S))$ a.s. Moreover, $u\in L^2(0,T;W^{1,2}(S))$ a.s.

\item If $u_0\in C^{1}(\overline{S})$ a.s., then the solution $u$ is in
$C^{\lambda}([0,T];C^{2\delta}(S))$ for all $\lambda,\delta>0$ such that
$\lambda+\delta<\frac12$. In particular, $u\in
C^{\beta_1,\beta_2}(\overline{S}\times [0,T])$  for all $\beta_1\in (0,1)$ and
$\beta_2\in (0,\frac12)$.
\end{enumerate}

If in (1) $u_0\in L^r(\O;L^p(S))$ for some $r\in (2, \infty)$, then
also
\[\E \sup_{t\in [0,T]}\|u(t)\|_{L^p(S)}^r\lesssim \E\|u_0\|_{L^p(S)}^r.\]
\end{example}

This example improves \cite[Theorem 3]{SSV} in several ways:
\begin{remark}\label{rem:compSSV}
The assumptions on the coefficients $a_{ij}$ and the domain $S$ we
have made are weaker than the ones in \cite[page 705]{SSV}. The
initial value in \cite{SSV} is assumed to be more regular than ours
(i.e.\ $u_0\in C^{2+\alpha}(\overline{S})$ instead of
$C^{1}(\overline{S})$) and it has to fulfill the Neumann boundary
condition at $t=0$. We consider $f$ and $g$ also depending on
$[0,T]\times \O\times S$. In \cite[Theorem 3]{SSV} the obtained
regularity is $C^{\beta_1,\beta_2}(\overline{S}\times [0,T])$ for
all $\beta_1\in (0,\alpha)$ and $\beta_2\in
(0,\frac{\alpha}{2}\wedge \frac{2}{n+2})$. Here $\alpha\in (0,1)$ is
a parameter which states how regular the coefficients $a_{ij}$ and
the domain $S$ are. Even in the limiting case $\alpha=1$, our time
regularity is better and it does not depend on the dimension
$n$.
\end{remark}

The condition on the noise term in \cite{SSV} is formulated as
\eqref{eq:condcoven} below.
\begin{remark}\label{rem:noiseSSV}
Since $Q$ is compact and positive, we can always find positive
numbers $(\lambda_n)_{n\geq 1}$ and an orthonormal system
$(e_n)_{n\geq 1}$ in $L^2(S)$ with $\sqrt{Q} = \sum_{n\geq
1}\lambda_n e_n \otimes e_n$. It follows that we may decompose $W$
as
\[W(t,s) = \sum_{n\geq 1} \sqrt{\lambda_n} W_n(t) e_n(s).\]
Here $(W_n)_{n\geq 1}$ are independent real-valued standard Brownian
motions.

The condition $\sqrt{Q}\in\calL(L^2(S),L^\infty(S))$ is for instance satisfied
if $(e_n)_{n\geq 1}$ in $L^\infty(S)$ and
\begin{equation}\label{eq:condcoven}
\sum_{n\geq 1} \lambda_n \|e_n\|_{L^\infty(S)}^2<\infty.
\end{equation}
Indeed, for all $h\in L^2(S)$, by the Cauchy-Schwartz inequality
\begin{equation}\label{eq:condcovencheck}
|\sqrt{Q} h(s)| = \Big|\sum_{n\geq 1} \sqrt{\lambda_n} e_n(s) [e_n,h]_{L^2(S)}
\Big|\leq \Big(\sum_{n\geq 1} \lambda_n |e_n(s)|^2\Big)^{\frac12} \,
\|h\|_{L^2(S)}<\infty
\end{equation}
for almost all $s\in S$.
\end{remark}

\begin{proof}[Proof of Example \ref{ex:Lip1}]
Let $E=L^p(S)$ with $p\in [2, \infty)$. Then conditions (AT1) and (AT2) are
satisfied (cf.\ \cite{MR945820, Schn,Ya}). Further, \ref{as:isom} is satisfied
with $\eta_0=1$ and (cf.\ \cite[Theorem 4.3.1.2]{Tr1})
\[\tE_{\eta} :=(L^p(S), W^{2,p}(S))_{\eta,2} =B^{2\eta}_{p,2}(S)\]
for $\eta\in (0,1]$ and $\tE_0 = E$. Note that these spaces are all UMD spaces with type
$2$ as follows from the explanation after \eqref{eq:type2est}.

%
%


Let $F:[0,T]\times\O\times E\to E$ be defined by $F(t,\omega,x)(s) =
f(t,\omega, s, x(s))$. Then $F$ satisfies \ref{as:LipschitzFtype}. Let
$B:[0,T]\times\O \times E\to \g(L^2(S),E)$ be defined as
\[(B(t,\omega,x) h)(s) = b(t,\omega,s,x(s)) (\sqrt{Q} h)(s).\]
This is well-defined by the assumptions and it satisfies
\ref{as:LipschitzBtype}. Indeed, under condition \eqref{eq:condCov}, we obtain
from Lemma \ref{lem:sq-fc-Lp} that for $x\in L^p(S)$,
\[
\|x \sqrt{Q}\|_{\g(L^2(S), L^p(S))}\lesssim_p K \|x\|_{L^p(S)}.
\]
and therefore, for $x_1, x_2\in L^p(S)$,
\begin{align*}
\|B(t,\omega,x_1)- B(t,\omega,x_2)\|_{\g(L^2(S), L^p(S))}& \lesssim_{p} K \|x_1-x_2\|_{L^p(S)},  \  t\in [0,T], \ \omega\in \O,
\\
\|B(t,\omega,x) h\|_{\g(L^2(S),L^p(S))} & \leq K (1+\|x\|_{L^p(S)})
, \ t\in [0,T], \ \omega\in\O.
\end{align*}

By Theorem \ref{thm:mainexistencegeninitial} with $a=\theta_F=\theta_B=0$, we
obtain that there exists a unique mild solution $U$ with paths in $C([0,T];E)$
a.s.

Next we use Proposition \ref{prop:varmild} to show that $U$ is also the unique
variational solution in $C([0,T];E)$. Note that condition \ref{eq:condW} is
satisfied since $A(t)$ is self-adjoint in the sense that $A(t)^*$ on $L^p(S)$
is $A(t)$ on $L^{p'}(S)$. Therefore, (AT2) holds for $A(t)^*$ and thus
\ref{eq:condW} holds by Remark \ref{rem:condW}. The result now follows from
Proposition \ref{prop:varmild}.

We still need to show that $U\in L^2(0,T;H^{1}(S))$ a.s.\ if $u_0\in
L^2(S)$ a.s.\ Let $E=L^2(S)$. It follows from Remark
\ref{rem:Hinfty} that $(A(t))_{t\in [0,T]}$ satisfies ($H^\infty$).
Since $A(t)$ is associated to a quadratic form with $V=W^{1,2}(S)$,
it follows that $D((-A_w(t))^{\frac12})=W^{1,2}(S)$ for $w$ large
enough with constants uniformly in time (cf.\ \cite[Section
2.2]{Ta1}). We have already shown that $U\in C([0,T];E)$ a.s.
Clearly, $B(U)$ is an element of $L^\infty([0,T];\g(L^2(S),E))$ and
by Theorem \ref{thm:maxreg}, $P\diamond B(U)\in L^2(0,T;W^{1,2}(S))$
a.s. For the deterministic convolution, it follows from
\eqref{eq:detconv} that $P*F(U)\in L^2(0,T;W^{1,2}(S))$. Finally, by
\eqref{eq:resVx}
\[\int_0^T \|P(t,0) u_0\|_{W^{1,2}(S)}^2\,ds \lesssim \int_0^T \|(-A_w(t))^{\frac12}P(t,0) u_0\|_{L^2(S)}^2\,dt\lesssim \|u_0\|_{L^2(S)}^2.\]
This completes the proof.

(2): \ Let $E=L^p(S)$ for $p\in[2,\infty)$. If $u_0\in C^1(\overline{S})$ a.s.,
then we claim that $u_0\in E_{b}^0$ a.s. for all $b\in [0,\frac12)$. Indeed, it
suffices to show that $u_0\in [E,D(A(0))]_{\frac12}$. By \cite[Theorem 7.2 and Remark 7.3]{Amnonhom} and \cite[Theorem 2.3]{DDHPV} (also see Example \ref{ex:boundary}), one has
\[[E,D(A(0))]_{\frac12} = [L^p(S), W^{2,p}(S)]_{\frac12} = W^{1,p}(S).\]
Since $C^{1}(\overline{S})\hookrightarrow W^{1,p}(S)$, the claim follows.

By Theorem \ref{thm:mainexistencegeninitial} the process $U$ has the following
regularity property: $U\in C^\lambda([0,T];E_{\delta})$ a.s.\ for all
$\lambda,\delta>0$ such that $\lambda+\delta<\frac12$. In particular taking $p$
large it follows from \cite[Theorem 4.6.1(e)]{Tr1} that $U\in
C^{\lambda}([0,T];C^{2\delta}(S))$ for all $\lambda,\delta>0$ such that
$\lambda+\delta<\frac12$.

The final assertion follows from
\eqref{eq:mainthmestimateholdertype}.
\end{proof}

Let us show that the variational solution of Example \ref{ex:Lip1} is also a
variational solution of the second type as defined in \cite{SSV}.
\begin{remark}\label{rem:variationex}
The variational solution of Example \ref{eq:condCov} satisfies: for all $t\in
(0,T]$, $\varphi\in C^{1}([0,t];L^2(S))$ such that $A(r,\cdot ,D) \varphi\in
C^{1}([0,t];L^2(S))$, a.s.
\begin{align*}
\int_S u(t,s)& \varphi(t,s) \, ds - \int_S u_0(s) \varphi(0,s) \, ds \\ & =
\int_0^t \int_S u(r,s) \varphi'(r,s) \, ds \, dr+ \int_0^t \int_S u(r,s)
A(r,s,D) \varphi(r,s) \, ds \, dr  \\ & \qquad +  \int_0^t \int_S  f(r,s,u(r,s)) \varphi(r,s) \, ds \, dr \\
& \qquad + \sum_{n\geq 1} \int_0^t \int_S b(r,s,u(r,s)) e_n(s) \varphi(r,s) \,
d W_n(r).
\end{align*}
Therefore, by integration by parts and approximation it follows that for all
$t\in (0,T]$, $\varphi\in W^{1,2}((0,t)\times S)$, a.s.
\begin{align*}
\int_S &u(t,s) \varphi(t,s) \, ds - \int_S u_0(s) \varphi(0,s) \, ds \\ & =
\int_0^t \int_S u(r,s) \varphi'(r,s) \, ds \, dr- \int_0^t \int_S \lb \nabla
u(r,s), a(r,s) \nabla \varphi(r,s)\rb_{\R^n} \, ds \, dr  \\ & \qquad +  \int_0^t \int_S  f(r,s,u(r,s)) \varphi(r,s) \, ds \, dr \\
& \qquad + \sum_{n\geq 1} \int_0^t \int_S b(r,s,u(r,s)) \sqrt{Q} e_n(s)
\varphi(r,s) \, d W_n(r).
\end{align*}
This coincides with the variational solution of the second kind from
\cite{SSV}.

\end{remark}

In the next example we will weaken the assumption on the covariance $Q$.

\begin{example}\label{ex:Lip2}
Consider equation \eqref{neumann} again. Assume the same conditions as in
Example \ref{ex:Lip1}, but with \eqref{eq:condCov} replaced by: there exist
$\beta\in (0,\frac12)$ and $q\in (\frac{n}{1-2\beta},\infty)$
\begin{equation}\label{eq:condCovb}
\sqrt{Q}\in \calL(L^2(S),L^q(S)).
\end{equation}

The following statements hold:
\begin{enumerate}
\item Let $p\in [2,\infty)$ be such that $p>(n^{-1}-q^{-1})^{-1}$. If $u_0\in
L^p(S)$ a.s., then there exists a unique mild and variational solution $u$ of
\eqref{neumann} with paths in $C([0,T];L^p(S))$ a.s.


\item If $u_0\in C^{1}(\overline{S})$ a.s., then the solution $u$ is in
$C^{\lambda}([0,T];C^{2\delta}(S))$ for all $\lambda,\delta>0$ such that
$\lambda+\delta<\beta$. In particular, $u\in
C^{\beta_1,\beta_2}(\overline{S}\times [0,T])$  for all $\beta_1\in (0,2\beta)$
and $\beta_2\in (0,\beta)$.
\end{enumerate}

\end{example}

This example improves \cite[Theorem 4]{SSV} in similar ways as
explained in Remark \ref{rem:compSSV}. Their condition on the noise
term is formulated as \eqref{eq:condcoven2} below.

\begin{remark}
Assume that $Q$ is compact and has the same form as in Remark
\ref{rem:noiseSSV}. The condition $\sqrt{Q}\in\calL(L^2(S),L^q(S))$
is for instance satisfied if $(e_n)_{n\geq 1}$ in $L^q(S)$ and
\begin{equation}\label{eq:condcoven2}
\sum_{n\geq 1} \lambda_n \|e_n\|_{L^q(S)}^2<\infty.
\end{equation}
Indeed, without loss of generality we may assume that $q>2$. Taking the
$L^{q}(S)$ norm on both sides in \eqref{eq:condcovencheck} yields
\begin{align*}
\|\sqrt{Q} h\|_{L^q(S)} &\leq \Big \|\Big(\sum_{n\geq 1} \lambda_n
|e_n(s)|^2\Big)^{\frac12}\Big\|_{L^q(S)} \, \|h\|_{L^2(S)} \\ & \leq
\Big(\sum_{n\geq 1} \lambda_n \|e_n\|_{L^q(S)}^2\Big)^{\frac12} \|h\|_{L^2(S)}
<\infty.
\end{align*}
\end{remark}

\begin{remark}
We should note that it is stated in \cite[Theorem 4 with
$\alpha=1$]{SSV} that the space regularity of the solution becomes
$C^{\sigma}(S)$ for all $\sigma<1$. We could not follow this
argument. It seems that for the definition of $Y_{\delta}$ in
\cite[Lemma 4]{SSV} one has restrictions on their parameter $\delta$
in terms of the $\beta$ from \eqref{eq:condCovb}.

For example consider the case that $S=(0,1)$, $A= \frac{d^2}{d s^2}$ with
Neumann boundary conditions, $f=0$, $b(x) = x$ and the noise is of the form
$W(t,x) = e_1(x) W_1(t)$, where $e_1\in L^q(S)$ and $W_1$ is a standard
Brownian motion. We do not believe that the solution has space regularity
$C^{\sigma}(S)$ for all $\sigma<1$, in general.
\end{remark}

\begin{proof}[Proof of Example \ref{ex:Lip2}]
We proceed as in Example \ref{ex:Lip1} but due to \eqref{eq:condcoven2} we need
to take $\theta_B>0$.

(1): \ Let $E=L^p(S)$. Since $Q\in \calL(L^2(S))$ we can assume that
$q\geq 2$. Let $r\in (1,\infty)$ be such that
$r(\frac{1}{p}+\frac{1}{q})=1$. Let $\theta_B\in
(\frac{n}{2r},\frac12)$. This is possible by the restriction on $p$.

Let $w\in \R$ be so large that $\lambda\in \rho(A_w)$ for all
$\text{Re}(\lambda)\leq 0$. We claim that for $x\in L^p(S)$ and
$h\in L^2(S)$,
\[\|(-A_w(t))^{-\theta_B} x \sqrt{Q} h \|_{L^\infty(S)} \lesssim \|x\|_{L^p(S)} \|h\|_{L^2(S)} \]
with constants uniformly in $t\in [0,T]$. Indeed, fix $\theta_B'\in
(\frac{n}{2r},\theta_B)$. By \cite[Theorem 4.6.1(e)]{Tr1} it follows that
\[\|y\|_{L^\infty(S)} \lesssim \|y\|_{B^{2\theta_B'}_{r,2}(S)}, \ \ y\in B^{2\theta_B}_{r,2}(S).\]
Moreover,
\[D((-A_w(t))^\theta_B) \hookrightarrow (L^r(S),D(A(t)))_{\theta_B',2}\hookrightarrow (L^r(S),W^{2,r}(S))_{\theta_B',2} =  B^{2\theta_B'}_{r,2}(S)\]
with embedding constants independent of $t\in [0,T]$. Here $D(A(t))$ stands
for the domain of $A(t)$ in $L^r(S)$ and similarly for the fractional domain
space. Therefore,
\begin{equation}\label{eq:embtheta}
\|A^{-\theta_B}(t) y\|_{L^\infty(S)}\lesssim \|y\|_{L^r(S)}, \ y\in L^r(S).
\end{equation}
From this and H\"older's inequality we obtain that
\begin{align*}
\|(-A_w(t))^{-\theta_B} x \sqrt{Q} h \|_{L^\infty(S)} & \lesssim \|x \sqrt{Q} h
\|_{L^r(S)}
\\ & \leq \|\sqrt{Q}\|_{\calL(L^2(S), L^q(S))} \|x\|_{L^p(S)} \|h
\|_{L^2(S)}.
\end{align*}

The claim and Lemma \ref{lem:sq-fc-Lp} imply that
\[\|(-A_w(t))^{-\theta_B} x \sqrt{Q}\|_{\g(L^2(S),L^p(S))} \lesssim \|\sqrt{Q}\|_{\calL(L^2(S), L^q(S))} \|x\|_{L^p(S)}.\]
It follows that there exists a constant $K$ such that for all
$x,y\in L^p(S)$ and for all $t\in [0,T], \omega\in\O$,
\[\|(-A_w(t))^{-\theta_B} (B(t,\omega,x) - B(t,\omega,y))\|_{\g(L^2(S),L^p(S))} \leq K \|x-y\|_{L^p(S)},\]
\[\|(-A_w(t))^{-\theta_B} B(t,\omega,x)\|_{\g(L^2(S),L^p(S))} \leq K (1+\|x\|_{L^p(S)}).\]

By Theorem \ref{thm:mainexistencegeninitial} (1) we obtain that there exists a
unique mild solution $u$ with paths in $C([0,T];L^p(S))$. The fact that $u$ is
also the unique variational solution follows in the same way as Example
\ref{ex:Lip1}.


(2): \ Let $\lambda,\delta>0$ be such that $\lambda+\delta<\beta$. Let
$\delta,\lambda>0$ be such that $\delta+\lambda<\beta$. Let $\delta'>\delta$ be
such that $\delta'+\lambda<\beta$. Choose $p\in [2,\infty)$ so large and
$\theta_B>\frac{n}{2r} = \frac{n}{2}(\frac1p+\frac1q)$ such that $\beta<\frac12
-\theta_B$. 

As in Example \ref{ex:Lip1} one has $u_0\in E_{\delta'+\lambda}^0$. By Theorem
\ref{thm:mainexistencegeninitial} (2) we obtain that $u$ has a version with
paths in $C^{\lambda}([0,T];B^{2\delta'}_{p,2}(S))$. By \cite[Theorem
4.6.1(e)]{Tr1} $B^{2\delta'}_{p,2}(S)\hookrightarrow C^{2\delta''-\frac{n}{p}}$
where $\delta<\delta''<\delta'$. Choosing $p$ large enough gives the result.
\end{proof}

As a final example we consider again \eqref{neumann}, but this time
with locally Lipschitz coefficients $f$ and $b$.

\begin{example}\label{ex:loclLip}
Consider equation \eqref{neumann}. Assume that $f,g:[0,T]\times\O\times
S\times\R\to \R$ are measurable, adapted and $f$ and $g$ are locally Lipschitz
in the fourth variable uniform in the others, i.e.\ for all $R>0$, there exists
$L_{f,R}$ and $L_{g,R}$ such that for all $t\in [0,T]$, $\omega\in\O$, $s\in
\R$ and $x,y\in\R$ with $|x|,|y|\leq R$,
\begin{eqnarray}\label{eq:Lipfloc}
|f(t,\omega,s,x) - f(t,\omega,s,y)|&\leq &L_{f,R} |x-y|, \ t\in [0,T],\omega\in
\O,s\in S,
\\ \label{eq:Lipgloc}
|g(t,\omega,s,x) - g(t,\omega,s,y)| & \leq & L_{g,R} |x-y|, \ t\in
[0,T],\omega\in \O,s\in S.
\end{eqnarray}
Assume that $A,C$ and $Q$ are as in Example \ref{ex:Lip1}. The following
statements hold:
\begin{enumerate}
\item Let $p\in (2n,\infty)$. Let $a\in (\frac{n}{p},\frac12)$. If $u_0\in
B^a_{p,p}(S)$ a.s., then there exists a unique maximal local mild solution
$(u(t))_{t\in [0,\tau)}$ of \eqref{neumann} with paths in
$C([0,\tau);B^{2a}_{p,p}(S))$ a.s.

\item If $u_0\in C^{1}(\overline{S})$ a.s., then the solution $u$ is in
$C^{\lambda}([0,T];C^{2\delta}(S))$ for all $\lambda,\delta>0$ such that
$\lambda+\delta<\frac12$. In particular, $u\in
C^{\beta_1,\beta_2}(\overline{S}\times [0,T])$  for all $\beta_1\in (0,1)$ and
$\beta_2\in (0,\frac12)$.
\end{enumerate}

If $f$ and $g$ are also of linear growth, i.e.\ \eqref{eq:linf} and
\eqref{eq:ling}, then the following hold:
\begin{enumerate}
\item[(1)$'$] Let $p\in (2n,\infty)$. Let $a\in (\frac{n}{p},\frac12)$. If
$u_0\in B^a_{p,p}(S)$ a.s., then $\tau=T$ and the solution $u$ from above is
the unique global mild and variational solution of \eqref{neumann} with paths
in $C([0,T];B^a_{p,p}(S))$ a.s.

\item[(2)$'$] If $u_0\in C^{1}(\overline{S})$ a.s., then the solution $u$ is in
$C^{\lambda}([0,T];C^{2\delta}(S))$ for all $\lambda,\delta>0$ such that
$\lambda+\delta<\frac12$. In particular, $u\in
C^{\beta_1,\beta_2}(\overline{S}\times [0,T])$  for all $\beta_1\in (0,1)$ and
$\beta_2\in (0,\frac12)$.
\end{enumerate}
\end{example}

\begin{remark} \
\begin{enumerate}
\item If $Q$ is as in Example \ref{ex:Lip2}, then one can still give
conditions under which existence, uniqueness and regularity hold.
This is left to the reader.

\item It is an interesting question under what conditions on $f$ and
$g$ different as \eqref{eq:linf} and \eqref{eq:ling}, one still
obtains a global solution. There are many results and approaches in
this direction. We refer the reader to \cite{MZ} and references
therein. We believe it is important to extend the ideas from
\cite{MZ} to our general framework. This could lead to
new global existence results.
\end{enumerate}
\end{remark}

We turn to the proof of Example \ref{ex:loclLip}. The set-up is
similar as in Example \ref{ex:Lip1}, but we need that $a>0$ to be
able to consider the locally Lipschitz coefficients $f$ and $b$.
Here $a$ is the parameter from Theorem
\ref{thm:mainexistenceLocallyLipschitzCasetype} which is used for
the underlying space $\tE_a$. The main reason we want $a>0$ is that
$\tE_a\hookrightarrow C(\overline{S})$ is needed.

\begin{proof}[Proof of Example \ref{ex:loclLip}]

(1):\ By \cite[Theorem 4.6.1(e)]{Tr1} it follows that
$\tE_a\hookrightarrow C(\overline{S})$ since $a>\frac{n}{p}$. Let
$E$ and $A$ be as in Example \ref{ex:Lip1}. For $0< \eta\leq 1$ let
\begin{equation}\label{eq:tEtapp}
\tE_{\eta} :=(L^p(S), W^{2,p}(S))_{\eta,p} =B^{2\eta}_{p,p}(S).
\end{equation}
It follows from \cite[Theorem 7.2 and Remark 7.3]{Amnonhom} and \cite[Theorem 2.3]{DDHPV} that for $2\eta\neq \frac1p$,
\begin{equation}\label{eq:tEtappass}
E_{\eta}^t := (E, D(A(t)))_{\eta, p} \hookrightarrow \tE_{\eta}\hookrightarrow
E
\end{equation}
with uniform constants in $t\in[0,T]$. Therefore, the version of
\ref{as:isom} explained below \ref{as:isom1a} on page
\pageref{as:isom1a} is satisfied except maybe for $2\eta = \frac1p$,
but this is not an actual problem since we can always take $\eta$
slightly larger in the above arguments. Note that by
\eqref{eq:tEtapp} and \eqref{eq:tEtappass}, $u_0\in E_a^0$ a.s.

Define $F:[0,T]\times\O\times \tE_a\to E$ by $F(t,\omega,x)(s) =
f(t,\omega, s, x(s))$. By \eqref{eq:Lipf} and $\tE_a\hookrightarrow
C(\overline{S})$, $F$ satisfies \ref{as:LipschitzFtypeloc}. Let
$B:[0,T]\times\O \times \tE_a\to \g(L^2(S),E)$ be defined as
\[(B(t,\omega,x) h)(s) = b(t,\omega,s,x(s)) (\sqrt{Q} h)(s).\]
By \eqref{eq:Lipg}, $\tE_a\hookrightarrow C(\overline{S})$, and the assumptions this is
well-defined and it satisfies \ref{as:LipschitzBtypeloc}.

By Theorem \ref{thm:mainexistenceLocallyLipschitzCasetype} with
$\theta_F=\theta_B=0$, we obtain that there exists a unique mild solution $U$
with paths in $C([0,\tau);\tE_a)$ a.s.

(2): \ Let $\lambda,\delta>0$ be such that $\lambda+\delta<\frac12$. Let $a>0$
be such that $\lambda+\delta+a<\frac12$ and let $p\in [2, \infty)$ be such that
$a>\frac{n}{p}$. Let $E$ and $\tE_a$ and $F,B$ etc.\ be as in (1). If $u_0\in
C^1(\overline{S})$ a.s., then as before one can show that $u_0\in E_{b}^0$ a.s.
for all $b\in [0,\frac12)$.


By Theorem \ref{thm:mainexistenceLocallyLipschitzCasetype} the process $U$ has
the following regularity property: $U\in C^\lambda([0,T];\tE_{a+\delta})$ a.s. In
particular it follows from \cite[Theorem 4.6.1(e)]{Tr1} that $U\in
C^{\lambda}([0,\tau);C^{2\delta}(S))$ for all $\lambda,\delta>0$ such that
$\lambda+\delta<\frac12$.

(1)$'$ and (2)$'$: \ This can be proved in the same way as (1) and (2), but now
using the linear growth assumption and the last part of Theorem
\ref{thm:mainexistenceLocallyLipschitzCasetype}.
\end{proof}

\appendix

\section{Technical proofs\label{sec:techproofs}}

Below we prove Proposition \ref{prop:varmild}. We recall it for convenience.

\begin{proposition}\label{prop:varmildapp}
Assume (AT), \ref{as:isom}, \ref{as:LipschitzFtype}, \ref{as:LipschitzBtype}
and \ref{eq:condW}. Let $r\in (2,\infty)$ be such that
$\theta_B<\frac12-\frac1r$. Let $U:[0,T]\times\O\to \tE_a$ be strongly measurable
and adapted and such that $U\in L^r(0,T;\tE_a)$ a.s. The following assertions are
equivalent:
\begin{enumerate}
\item $E$ is a mild solution of \eqref{SEtype}.

\item $U$ is a variational solution of \eqref{SEtype}.
\end{enumerate}
\end{proposition}
Condition \ref{eq:condW} from page \pageref{eq:condW} is only needed
in $(2)\Rightarrow (1)$.

\begin{proof}
$(1)\Rightarrow (2)$: \ Let
\[F_{-\theta_F}(r,x) = (-A_w(r))^{-\theta_F} F(r,x), \ \ B_{-\theta_B}(r,x) = (-A_w(r))^{-\theta_B} B(r,x)\]
and $P_{\theta}(t,r) = P(t,r) (-A_w(r))^{\theta}$ for $\theta = \theta_F$ or
$\theta = \theta_B$.

Let $t\in [0,T]$ be arbitrary and $\varphi\in \Gamma_t$. Since $U$ is a.s.\ in
$L^1(0,T;E)$ we have that $s\mapsto \lb U(s), A(s)^* \varphi(s) \rb$ is
integrable and from the definition of a mild solution we obtain that a.s.,
{\small
\begin{equation}\label{consmildsol}
  \begin{aligned}
    & \int_0^t \lb U(s), A(s)^*  \varphi(s)\rb \, ds \\ & = \int_0^t \lb P(s,0) u_0,
    A(s)^*\varphi(s)\rb \, ds
    + \int_0^t \int_0^s \lb P_{\theta_F}(s,r) F_{-\theta_F}(r,U(r)), A(s)^*
    \varphi(s) \rb \, d r \, ds
    \\ &  +  \int_0^t \int_0^s B_{-\theta_B}(r,U(r))^* P_{\theta_B}(s,r)^* A(s)^*
    \varphi(s) \, d W_H(r) \, ds.
  \end{aligned}
\end{equation}}

Since $(P(t,s))_{0\leq s\leq t\leq T}$ is an evolution family that solves
\eqref{nCP},
it follows from an approximation argument that for all $x\in E$ and $0\leq
r\leq t\leq T$,
\begin{equation}\label{weak1}
\begin{aligned}
  \lb P(t,r)  x, \varphi(t)\rb - & \lb x, \varphi(r)\rb \\ & = \int_r^t \lb P(s,r)
  x,A(s)^*\varphi(s)\rb \, ds + \int_r^t \lb P(s,r)x, \varphi'(s)\rb \, ds.
\end{aligned}
\end{equation}
Therefore, by another approximation argument we obtain that for all $\theta\in
[0,1)$ and for all $x\in E$ and $0\leq r\leq t$,
\begin{equation}\label{weak1a}
\begin{aligned}
  \lb P_{\theta}(t,r) &x, \varphi(t)\rb - \lb x, ((-A_w(r))^{\theta})^* \varphi(r)\rb \\ & = \int_r^t \lb
  P_{\theta}(s,r)
  x,A(s)^*\varphi(s)\rb \, ds + \int_r^t \lb P_{\theta}(s,r) x, \varphi'(s)\rb \, ds.
\end{aligned}
\end{equation}
As a consequence one obtains that for all $R\in \calL(H,E)$ and $0\leq r\leq
t$,
\begin{equation}\label{weak2}
\begin{aligned}
  R^* P_{\theta_B}&(t,r)^* \varphi(t)-  R^* ((-A_w(r))^{\theta_B})^* \varphi(r) \\ & = \int_r^t R^*
  P_{\theta_B}(s,r)^*
  A(s)^*\varphi(s) \, ds + \int_r^t R^* P_{\theta_B}(s,r)^* \varphi'(s) \, ds.
\end{aligned}
\end{equation}
Indeed, this follows from \eqref{weak1a} by applying $h\in H$ on both sides.

By the Fubini theorem and \eqref{weak1a} we obtain a.s.,
{\small
\begin{align*}
\int_0^t & \int_0^s \lb P_{\theta_F}(s,r) F_{-\theta_F}(r,U(r)), A(s)^*
\varphi(s) \rb \, dr \, ds
\\ & = \int_0^t \lb P_{\theta_F}(t,r) F_{-\theta_F}(r,U(r)), \varphi(t) \rb \, dr - \int_0^t \lb F_{-\theta_F}(r, U(r)), ((-A_w(r))^{\theta_F})^*\varphi(r) \rb \, dr
\\ & \qquad - \int_0^t \int_0^s \lb P_{\theta_F}(s,r) F_{-\theta_F}(r,U(r)), \varphi'(s) \rb \, dr \, ds.
\end{align*}}
By the stochastic Fubini theorem and \eqref{weak2} we obtain that a.s.,
{\small
\begin{align*}
\int_0^t &\int_0^s B_{-\theta_B}(r,U(r))^*
P_{\theta_B}(s,r)^* A(s)^* \varphi(s) \, d W_H(r) \, ds =
\\ & = \int_0^t  B_{-\theta_B}(r,U(r))^* P_{\theta_B}(t,r)^* \varphi(t) \, d W_H(r) \\ & \qquad - \int_0^t
B_{-\theta_B}(r,U(r))^* ((-A_w(r))^{\theta_B})^*\varphi(r) \, d W_H(r) \\
& \qquad - \int_0^t \int_r^t B_{-\theta_B}(r,U(r))^* P_{\theta_B}(s,r)^*
\varphi'(s) \, d W_H(r) \, ds.
\end{align*}}
Therefore, it follows from \eqref{consmildsol}, \eqref{weak1} and the
definition of a mild solution that
\begin{align*}
\int_0^t \lb U(s), &A(s)^* \varphi(s)\rb \, ds =  \lb U(t),
\varphi(t) \rb - \int_0^t \lb U(s), \varphi'(s) \rb \, ds- \lb u_0,
\varphi(0)\rb \\ & \quad - \int_0^t  \lb F_{-\theta_F}(r,U(r)),
((-A_w(r))^{\theta_F})^*\varphi(r)\rb \, d r \\ & \quad - \int_0^t
B_{-\theta_B}(r,U(r)) ((-A_w(r))^{\theta_B})^* \varphi(r) \, d W_H(r)
\end{align*}
and we obtain that $U$ is a variational solution.

$(2)\Rightarrow (1)$: \
Let $t\in [0,T]$ be arbitrary. We show that for all $x^*\in \Upsilon_t$, a.s.
\begin{equation}\label{mildsol}
\begin{aligned}
\lb U(t), x^* \rb & = \lb P(t,0) u_0, x^*\rb + \int_0^t \lb P_{\theta_F}(t,s)
F_{-\theta_F}(s,U(s)), x^*\rb \,d s
\\ &  + \int_0^t B_{-\theta_B}(s, U(s)))^* P_{\theta_B}(t,s)^{*} x^* \,dW_H(s).
\end{aligned}
\end{equation}
By the existence of the integral, the existence of the stochastic integral, the
weak$^*$-sequential density of $\Upsilon_t$ (see \ref{eq:condW} and the
Hahn-Banach theorem this suffices. For $x^*\in \Upsilon_t$, let $\varphi(s) =
P(t,s)^* x^*$. Then it follows from \eqref{varsol} and \eqref{varphiMinA} that
\begin{align*}
  \lb & U(t), x^*\rb - \lb P(t,0) u_0, x^*\rb + \int_0^t \lb U(s), A(s)^*
  P(t,s)^*x^* \rb \, ds \\ & = \int_0^t \lb U(s), A(s)^*P(t,s)^*x^* \rb \, ds +
  \int_0^t \lb F_{-\theta_F}(s, U(s)), ((-A_w(s))^{\theta_F})^* P(t,s)^{*} x^*\rb \, d s
  \\ & \qquad + \int_0^t B_{-\theta_B}(s, U(s))^* ((-A_w(s))^{\theta_B})^* P(t,s)^{*} x^* \, d W_H(s)
\end{align*}
and we may conclude \eqref{mildsol}.
\end{proof}

{\em Acknowledgment} -- The author is grateful to Roland Schnaubelt and Lutz
Weis for helpful discussions. Moreover, he thanks the anonymous referees for
carefully reading the manuscript and for giving many useful comments.

\def\cprime{$'$}
\providecommand{\bysame}{\leavevmode\hbox to3em{\hrulefill}\thinspace}

\end{document}